\documentclass[11pt]{amsart}

\usepackage{mathdots}
\usepackage{xy}
\usepackage{array}
\usepackage{amssymb}
\usepackage{psfrag}
\usepackage{graphics}
\usepackage{graphicx}
\usepackage{epsfig}

\usepackage{epic}
\usepackage{subfigure}

\xyoption{matrix}
\xyoption{all}

\newtheorem{lemma}{Lemma}[section]
\newtheorem{theorem}[lemma]{Theorem}
\newtheorem{corollary}[lemma]{Corollary}

\newtheorem{proposition}[lemma]{Proposition}

\newtheorem{remark}[lemma]{Remark}

\newtheorem{definition}[lemma]{Definition}

\newcommand{\Hom}{\operatorname{Hom}}
\newcommand{\Ext}{\operatorname{Ext}}
\newcommand{\Gen}{\operatorname{Gen}}
\newcommand{\Cogen}{\operatorname{Cogen}}

\newcommand{\eperp}{{\perp}_{\scriptscriptstyle E}}
\newcommand{\hperp}{{\perp}_{\scriptscriptstyle H}}

\newcommand{\module}{\operatorname{mod}}
\newcommand{\Module}{\operatorname{Mod}}

\newcommand{\T}{\mathbf T}
\newcommand{\Torsion}{\mathcal T}
\newcommand{\torsion}{T}
\newcommand{\Free}{\mathcal F}
\newcommand{\free}{F}
\newcommand{\rigid}{U}

\newcommand{\C}{\mathcal C}
\newcommand{\W}{\mathcal W}
\newcommand{\R}{\mathcal R}
\renewcommand{\SS}{\mathcal S}
\newcommand{\X}{\mathcal X}
\newcommand{\A}{\mathcal A}

\newcommand{\AAA}{\mathbb A}
\newcommand{\UUU}{\mathbb U}

\newcommand{\ind}{\operatorname{ind}}
\newcommand{\add}{\operatorname{add}}

\begin{document}
\title{Torsion pairs and rigid objects in tubes}

\begin{abstract}
We classify the torsion pairs in a tube category and show that
they are in bijection with maximal rigid objects in the extension of the tube
category containing the Pr\"{u}fer and adic modules. We show that the annulus
geometric model for the tube category can be extended to the larger category
and interpret torsion pairs, maximal rigid objects and the bijection between
them geometrically. We also give a similar geometric description in the case of
the linear orientation of a Dynkin quiver of type A.
\end{abstract}

\author[Baur]{Karin Baur}
\address{
Institut f\"ur Mathematik und wissenschaftliches Rechnen \\
Universit\"at Graz \\
Heinrichstrasse 36 \\
A-8010 Graz \\
Austria
}
\email{baurk@uni-graz.at}
\author[Buan]{Aslak Bakke Buan}
\address{
Department of Mathematical Sciences \\
Norwegian University of Science and Technology \\
N-7491 Trondheim \\
NORWAY
}
\email{aslakb@math.ntnu.no}

\author[Marsh]{Robert J. Marsh}
\address{School of Mathematics \\
University of Leeds \\
Leeds LS2 9JT \\
England
}
\email{marsh@maths.leeds.ac.uk}

\thanks{This work was supported by the Engineering and Physical Sciences
Research Council [grant number EP/G007497/1], by the FIM (Institute for
Mathematical Research) at the ETH, Z\"{u}rich, by a grant from the
Research Council of Norway, FRINAT grant number 196600 and
by the SNSF, grant number PP0022-114794.}

\date{2 November 2012}


\maketitle

\section*{Introduction}

Torsion pairs in abelian categories were introduced by Dickson~\cite{dickson66}
(see the introduction to~\cite{br} for further details). They play an important
role in the study of localisation (see e.g.~\cite{br}) and in tilting theory
(see e.g.~\cite{ass,br}). Also, torsion pairs in the case of triangulated
categories, as defined in~\cite[2.2]{iy}, have been considered recently by a
number of authors, e.g.~\cite{hjr,iy,kz,nakaoka,ng,zz}.

The main object of study of this article is the collection of torsion pairs in
a tube category, which is a hereditary abelian category.
Such categories arise as full subcategories of module
categories over tame hereditary algebras, and are so-called because their
Auslander-Reiten quivers have the shape of a tube.

A geometric model for tube categories has been given in~\cite{bama,warkentin}
(see also~\cite{bz}). The indecomposable objects are parametrized by a
collection of oriented arcs in an annulus with $n$ marked points on one of its
boundary components. The dimension of the $\Ext$-group between two
indecomposable objects coincides with the negative intersection number of
the corresponding pair of arcs.

In this article, we classify the torsion pairs in a tube category $\T$. We build
on results in~\cite{bk1} giving a bijection between torsion pairs in a tube category and
cotilting objects in the category obtained by taking arbitrary direct limits of modules in the tube.
We show that all torsion pairs in the tube category arise in this way or via a dual construction.
Thus they are in bijection with maximal rigid objects in the category $\overline{\T}$
obtained from the tube by taking arbitrary direct or inverse limits of objects in the tube.
We give an explicit description of this bijection and its inverse.

We further show that the geometric model referred to above can be extended to the indecomposable
objects in $\overline{\T}$, i.e.\ to include Pr\"{u}fer and adic modules associated to the tube. These
extra objects are represented by certain infinite arcs in the annulus which spiral in towards the
inner boundary component. The result concerning the dimensions of $\Ext$-groups extends to this case also.
This enables us to give a characterisation of maximal rigid objects in $\T$ in the geometric model. We also
give a characterisation of torsion
pairs in the geometric model in terms of certain closure properties corresponding to closure properties
of the subcategories in a torsion pair. In particular, the collection of arcs corresponding to a
subcategory in a torsion pair must form an \emph{oriented Ptolemy diagram}, an oriented version in the
annulus case of the Ptolemy diagrams appearing in~\cite{hjr,ng}, as well as satisfying additional criteria.
We also give a geometric description of the above bijection and its inverse.

In order to give the geometric interpretation, we first give a similar model
for the linearly oriented quiver in type A
(note that M. Warkentin~\cite{warkentin} also suggests such a model) and
show how to interpret tilting modules and torsion pairs in this model.

We remark that, in independent work, T. Holm, P. J\o rgensen and
M. Rubey~\cite{holmjorgensenrubey} have classified the torsion pairs in the
cluster category associated to a tube (as opposed to the tube itself,
which we study here). The torsion pairs in the cluster case are different,
although we note that unoriented Ptolemy diagrams play a role in the cluster
tube case, while oriented Ptolemy diagrams appear here.

In Section 1, we recall the definition and some of the properties of torsion pairs in abelian categories.
In Section 2, we discuss the type A case. In Section 3, we recall the geometric model of the tube and show how it can be extended to include Pr\"{u}fer and
adic modules. We also discuss a certain reflection map (from~\cite{bk1})
on the indecomposable objects of the tube, and its properties.
In Section 4 we classify the torsion pairs in the tube and prove that
they are in bijection with maximal rigid objects in $\overline{\T}$,
giving an explicit description of the bijection and its inverse.
In Section 5, we give a geometric interpretation of maximal rigid objects
in $\T$ and torsion pairs in $\T$ and the bijection between them.

\section{Preliminaries} \label{s:preliminaries}

We shall adopt the convention throughout that all subcategories considered
are strictly full (i.e.\ full and closed under isomorphism). We shall also
consider modules up to isomorphism.
Let $\A$ be an abelian category. If $\X$ is a subcategory of $\A$,
we define ${}^{\hperp} \X$ (respectively, ${}^{\eperp} \X$) to be
the additive subcategory of $\A$ consisting of the objects
$Y$ satisfying $\Hom_{\A}(Y,X)=0$ (respectively,
$\Ext^1_{\A}(Y,X)$=0) for all $X$ in $\X$.
For an object $X$ of $\A$ we define ${}^{\hperp}
X={}^{\hperp} \add X$, and ${}^{\eperp} X={}^{\eperp} \add X$,
where $\add X$ denotes the additive subcategory of $\A$ whose
objects are finite direct sums of direct summands of $X$. We
similarly define $\X^{\hperp}$, $\X^{\eperp}$, $X^{\hperp}$ and
$X^{\eperp}$. For additive subcategories $\X_i$, $i=1,2, \ldots ,m$,
of $\A$, we write $\coprod_{i=1}^m \X_i$ for the smallest additive
subcategory of $\A$ containing the $\X_i$.

If $\X$ is an additive subcategory of $\A$, we write $\Gen \X$ for the
subcategory consisting of objects which are quotients of objects of $\X$ and,
for an object $X$, $\Gen X=\Gen(\add X)$. Similarly, we write $\Cogen \X$ for
the subcategory consisting of objects which are subobjects of objects of $\X$
and $\Cogen X=\Cogen(\add X)$. We write $\ind \X$ for the subcategory of $\X$
whose objects are the indecomposable objects in $\X$.

We next recall some of the theory of torsion pairs. Recall that
a pair $(\Torsion,\Free)$ of subcategories of
$\A$ is said to be a \emph{torsion pair} \cite[Sect.\ 1]{dickson66}
if $\Hom(\torsion,\free)=0$ for all objects $\torsion$ in $\Torsion$ and $\free$ in $\Free$ and
for every object $A$ in $\A$ there is a short exact sequence
\begin{equation}
0 \rightarrow {\torsion} \rightarrow A \rightarrow {\free} \rightarrow 0
\label{e:linkingsequence}
\end{equation}
with $\torsion$ in $\Torsion$ and $\free$ in $\Free$.
We say that $\Torsion$ is the \emph{torsion} part and $\Free$ is the
\emph{torsion-free} part of the pair.

Recall that an abelian category $\A$ is said to be \emph{finite
length} if it is skeletally small and every object in it is artinian
and noetherian. Equivalently, it is skeletally small and every object
has a finite length composition series.

Thus, for example, the category $\module A$ of finite dimensional modules over a finite-dimensional algebra is a finite length category.
Parts (a), (b) of the following can be proved using arguments as
in~\cite[Thm.\ 2.1]{dickson66} (see also~\cite[Lemma 1.7]{bk1}); part (c) is
a well-known corollary.

\begin{lemma} \label{l:tpclosure}
Suppose that $\A$ is a finite length abelian category.
\begin{enumerate}
\item[(a)]
Let $\Torsion$ be a subcategory of $\A$ and let $\Free=\Torsion^{\hperp}$.
Then $(\Torsion,\Free)$ is a torsion pair in $\A$ if and only if
$\Torsion$ is closed under quotients and extensions.
\item[(b)]
Let $\Free$ be a subcategory of $\A$ and let $\Torsion={}^{\hperp}\Free$.
Then $(\Torsion,\Free)$ is a torsion pair in $\A$ if and only if $\Free$ is closed under
subobjects and extensions.
\item[(c)] A pair $(\Torsion,\Free)$
of subcategories of $\A$ is a torsion pair if and only if
$\Torsion^{\hperp}=\Free$ and ${}^{\hperp}\Free=\Torsion$.
\end{enumerate}
\end{lemma}

Recall that a finite length abelian category is said to be \emph{serial} if
every object is a direct sum of indecomposable uniserial objects. We recall
the following well-known result.

\begin{lemma} \label{l:uniserialclosure}
Let $\A$ be a serial finite length abelian category. Then
\begin{enumerate}
\item[(a)] A morphism from an indecomposable object of $\A$ to a direct sum of indecomposable
objects is a monomorphism if and only if at least one of its components is a monomorphism.
\item[(b)] Let $\X$ be a subcategory of $\ind \A$.
Then we have $\ind \Gen(\add \X)=\ind \Gen \X$ and $\ind \Cogen(\add \X)=\ind \Cogen \X$.
\end{enumerate}
\end{lemma}

Note that the dual of Lemma~\ref{l:uniserialclosure}(a) also holds.
We fix an algebraically-closed field $K$ and denote by $D$ the vector space duality $\Hom_K(-,K)$.
Next, consider a finite dimensional $K$-algebra $\Lambda$. We denote
by $\Module \Lambda$ the category of all left $\Lambda$-modules, and by
$\module \Lambda$ the subcategory of all finitely generated $\Lambda$-modules.
Recall that a module ${\rigid}$ in $\module \Lambda$ is called \emph{tilting} if
\begin{itemize}
\item[(a)] The projective dimension of ${\rigid}$ is at most 1;
\item[(b)] $\Ext^1({\rigid},{\rigid}) = 0$;
\item[(c)] There is a short exact sequence
$$0\rightarrow \Lambda \rightarrow {\rigid}_0 \rightarrow {\rigid}_1 \rightarrow 0$$
with ${\rigid}_0$ and ${\rigid}_1$ in $\add {\rigid}$.
\end{itemize}
Cotilting modules in $\module \Lambda$ are defined dually. We only
consider basic tilting or cotilting modules, i.e.\ we assume that
${\rigid}= \amalg {\rigid}_i$, with ${\rigid}_i$ indecomposable and ${\rigid}_i \not\simeq {\rigid}_j$
for $i \neq j$.

\section{Dynkin type A}
\label{s:Dynkin}

We consider a linearly oriented quiver $Q$ of type $\text{A}_n$:
$$
\xymatrix{
1 & 2 \ar[l] & 3 \ar[l] & \cdots\ar[l] & n-1 \ar[l] & n \ar[l]
}
$$
We write $S_i$, $i=1,\dots,n$, for the simple $KQ$-modules.
Let $M_{ij}$ denote the indecomposable $KQ$-module with composition factors
$S_{i+1},\dots,S_{j-1}$ (starting from the socle).
If $i$ lies in $\{j-1,j\}$ we regard $M_{ij}$ as zero.
Note that $\module KQ$ is serial.

\subsection{Geometric model}
\label{s:geometricmodel}

We now describe a geometric model for the indecomposable $KQ$-modules. Note that
a similar such model is also suggested in~\cite[Remark 4.28]{warkentin}.
Here we will indicate how this model can incorporate torsion theories
in $\module KQ$ --- later we will indicate how this can be done for tubes. We
also give some more explicit information as preparation for this.
We consider a line segment $\ell_n$ with marked points $0,1,\dots,n+1$:
\[
\begin{picture}(220.00,20.00)
\drawline(10.00,15.00)(210.00,15.00)
\put(10.00,15.00){\makebox(0,0)[cc]{\tiny $\bullet$}}
\put(10.00,05.00){\makebox(0,0)[cc]{\tiny $0$}}
\put(40.00,15.00){\makebox(0,0)[cc]{\tiny $\bullet$}}
\put(40.00,05.00){\makebox(0,0)[cc]{\tiny $1$}}
\put(70.00,15.00){\makebox(0,0)[cc]{\tiny $\bullet$}}
\put(70.00,05.00){\makebox(0,0)[cc]{\tiny $2$}}
\put(150.00,15.00){\makebox(0,0)[cc]{\tiny $\bullet$}}
\put(150.00,05.00){\makebox(0,0)[cc]{\tiny $n-1$}}
\put(180.00,15.00){\makebox(0,0)[cc]{\tiny $\bullet$}}
\put(180.00,05.00){\makebox(0,0)[cc]{\tiny $n$}}
\put(210.00,15.00){\makebox(0,0)[cc]{\tiny $\bullet$}}
\put(210.00,05.00){\makebox(0,0)[cc]{\tiny $n+1$}}
\end{picture}
\]
and associate the module $M_{ij}$ with the arc $[i,j]$ above $\ell_n$
from $i$ to $j$ oriented towards $j$, $0\le i<j-1\le n$.
This gives a bijection between indecomposable $KQ$-modules and the
set $\A(\ell_n)$ of arcs up to isotopy between marked points of
$\ell_n$, above $\ell_n$, which are not isotopic to boundary arcs.

It is easy to check that, for $[i,j]$, $[i',j']\in \A(\ell_n)$,
we have $\Ext^1(M_{ij},M_{i'j'})\cong K$ if there is a negative crossing between
$[i,j]$ and $[i',j']$ (see Figure~\ref{fig:neg}) and is zero otherwise.
In the former case, the non-split extension takes the form:
$$
0\to M_{i'j'}\to \begin{array}{c} M_{i'j} \\ \amalg \\ M_{ij'} \end{array} \to M_{ij} \to 0
$$
up to equivalence, if $j'>i+1$, while if $j'=i+1$, it takes the form
$$
0\to M_{i'j'}\to M_{i'j} \to M_{ij} \to 0,
$$
This is interpreted geometrically in Figure~\ref{fig:split-sum}
where the indecomposable summands of the middle term
of the short exact sequence are indicated by dotted lines.

\begin{figure}[ht]
\psfragscanon
\psfrag{i}{\tiny $i$}
\psfrag{i'}{\tiny $i'$}
\psfrag{j}{\tiny $j$}
\psfrag{j'}{\tiny $j'$}
\includegraphics[scale=.5]{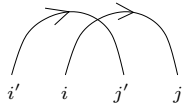}
\caption{A negative crossing between $[i,j]$ and $[i',j']$.}\label{fig:neg}
\end{figure}

\begin{figure}[ht]
\psfragscanon
\psfrag{i}{\tiny $i$}
\psfrag{i'}{\tiny $i'$}
\psfrag{j}{\tiny $j$}
\psfrag{j'}{\tiny $j'$}
\psfrag{j'=i+1}{\tiny $j'=i+1$}
\subfigure[Case $j'>i+1$]{
\includegraphics[scale=.6]{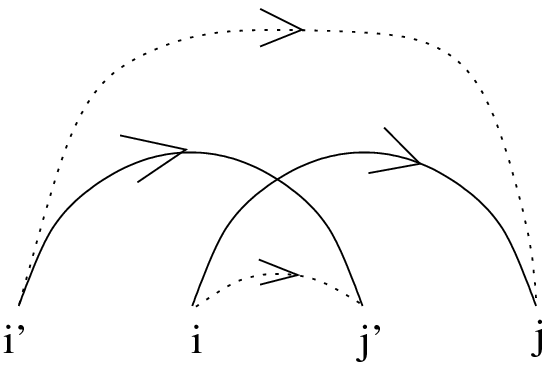}}
\ \ \ \ \ \
\subfigure[Case $j'=i+1$]{
\includegraphics[scale=.6]{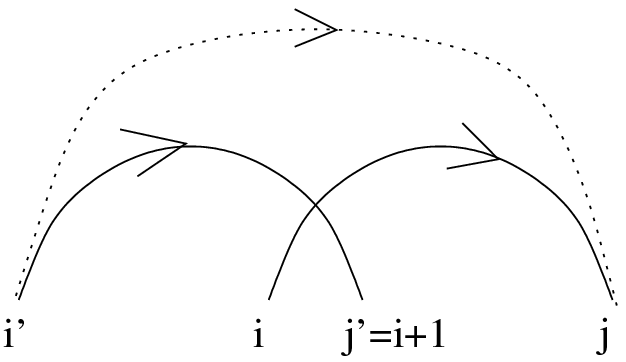}}
\caption{Non-split extension}\label{fig:split-sum}
\end{figure}

The AR-translation moves an arc one step to the left (or gives zero
if this is not defined). Furthermore, the indecomposable quotients
of a module correspond to moving the starting point weakly to the right,
while submodules correspond to moving the ending point weakly to the left.
We call the former arcs {\em left-shortenings} of $[i,j]$ and the latter arcs
{\em right-shortenings} of $[i,j]$.

By~\cite{bo}, a $KQ$-module is tilting if and only if it is
a maximal rigid $KQ$-module.
Since $M_{0,n+1}$ is projective-injective, it is a summand of every tilting module.
It follows that in the above model, tilting modules are in bijection with
triangulations of a polygon with $n+2$ sides.

Motivated by~\cite{ng,hjr} and the above description of extensions,
we call a collection of arcs in $\A(\ell_n)$ an {\em oriented Ptolemy
diagram} if, whenever $[i,j]$ and $[i',j']$ lie in the collection, with
$i'<i<j'<j$, we have that the arcs $[i,j']$ (provided $j'>i+1$) and $[i',j]$
also lie in the collection.

\subsection{Torsion pairs}

By Lemma~\ref{l:tpclosure} and Lemma~\ref{l:uniserialclosure},
a collection of arcs in $\A(\ell_n)$ corresponds to the torsion (respectively, torsion-free) part of a torsion pair in $\module KQ$ if and only if it forms an oriented Ptolemy diagram and is closed under left-shortening (respectively, right-shortening).

Given a tilting $KQ$-module ${\rigid}$, the pair
$(\Gen {\rigid}, \Cogen\tau {\rigid})=({\rigid}^{\eperp},{\rigid}^{\hperp})$
is known to form a torsion pair (see~\cite[VI.2]{ass}). A torsion pair arises
in this way if and only if $\Torsion$ contains all the indecomposable injective
modules, if and only if $\Free$ contains no non-zero injective module (see~\cite[VI.6]{ass}
for the first equivalence; the second is easy to check). The first equivalence
holds for an arbitrary finite dimensional hereditary algebra of finite representation type, and the
equivalence of the second two statements holds for any finite dimensional hereditary algebra.

Noting that a $KQ$-module is tilting if and only if it is cotilting, we have:

\begin{corollary}\label{cor:Gen-Cogen}
The map:
$$
{\rigid} \mapsto (\Gen {\rigid},\Cogen\tau {\rigid})=({\rigid}^{\eperp},{\rigid}^{\hperp})
$$
gives a bijection between tilting $KQ$-modules and
torsion pairs $(\Torsion,\Free)$ for which $\Torsion$ contains all the
indecomposable injective modules. The map:
$$
{\rigid} \mapsto (\Gen\tau^{-1}{\rigid},\Cogen {\rigid})=({}^{\hperp}{\rigid},{}^{\eperp}{\rigid})
$$
gives a bijection between tilting $KQ$-modules and
torsion pairs $(\Torsion,\Free)$ for which for which $\Free$ contains all
the indecomposable projective modules.
\end{corollary}

Using Lemma~\ref{l:uniserialclosure}, the first map
in Corollary~\ref{cor:Gen-Cogen} can be interpreted in the geometric model:
$\Gen {\rigid}$ is obtained from ${\rigid}$ by closure under left shortening and
$\Cogen\tau {\rigid}$ is obtained from ${\rigid}$ by shifting to the left one step (deleting arcs starting at $0$) and then closing under right shortening (in both cases we then
take the additive closure).

Conversely, if $(\Torsion,\Free)$ is a torsion pair of the kind considered in
Corollary~\ref{cor:Gen-Cogen}, then ${\rigid}$ can be recovered as the direct
sum of the indecomposable Ext-projectives in $\Torsion$, i.e.\ the objects
\begin{equation}
\{X\in \ind(\Torsion): \Ext^1(X,{\torsion})=0\ \text{\ for all\ } {\torsion}\in \Torsion\},
\label{e:extprojectives}
\end{equation}
by~\cite[VI.2.5]{ass}.
Geometrically, this means taking all of the arcs $X$ in
$\ind \Torsion$ which do not have a negative crossing with an arc in $\ind \Torsion$.

There is also a geometric description of the second map in Corollary~\ref{cor:Gen-Cogen}
and its inverse. We leave the details to the reader.

\section{Tubes}

\subsection{Categorical description}

Fix an integer $n\geq 1$. Consider the quiver $Q$
\begin{equation}\label{quivertilde}
\xymatrix@R=0.3cm{
& 2 \ar[dl] & 3 \ar[l] & \cdots \ar[l] & n \ar[l] & \\
1 & & & & & n+1 \ar[lllll] \ar[ul]
}
\end{equation}
of Euclidean type $\widetilde{\text{A}}_{1,n}$. The path algebra $\Lambda= KQ$ is tame hereditary, and the module category $\module KQ$ has an
extension closed subcategory $\T_n$, which can be realized as the extension closure of the modules $L, S_2, \dots, S_{n}$,
where $S_i$ denotes the simple corresponding to vertex $i$, and $L$ denotes the unique indecomposable module
with composition factors $S_1$ and $S_{n+1}$. The category $\T_n$ is called a {\em tube} of rank $n$. Note that the indecomposables of $\T$ form a standard component
(see e.g.~\cite{ass}) of the AR-quiver of $KQ$.

Actually $\T_n$ is a hereditary finite length abelian category with
$n$ simple objects, and equivalent categories appear in various
settings in representation theory and algebraic geometry.
For each pair of objects $X,Y$ in $\T_n$, the
spaces $\Hom(X,Y)$ and $\Ext^1(X,Y)$ have finite $K$-dimension, and
there is an autoequivalence $\tau$ on $\T_n$, induced by the
Auslander-Reiten translate on $\module \Lambda$, with the property
that $\Hom(Y, \tau X) \simeq D\Ext^1(X,Y)$. Let
$\sigma:\mathbb{Z}\rightarrow \mathbb{Z}$ be the map taking $i$ to
$i+n$. From now on we denote the simples in $\T = \T_n$ by
$M_{i,i+2}$ for $i=0, \dots, n-1$, in such a way that $\tau
M_{i,i+2} = M_{i-1,i+1}$,
where we regard $M_{\sigma^k(i),\sigma^k(j)}$ as equal to $M_{i,j}$
for any integer $k$.
The category $\T_n$ is serial; thus the indecomposable objects in
$\T_n$ are uniserial and uniquely determined by their simple socle
and length (in $\T_n$).
We denote by $M_{i,i+l+1}$ an indecomposable with socle $M_{i,i+2}$
and length $l$. Then $\tau M_{i,i+l+1}  = M_{i-1,i+l}$,
and the AR-quiver of the tube $\T_n$ is as in Figure~\ref{fig:ARquiverTn}
(with the columns on the left- and right-hand sides identified).
Note that each indecomposable object has a unique name $M_{ij}$
(with $j-i\geq 2$) if we insist that $i$ lies in $\{0,1,\ldots ,n-1\}$.

\begin{figure}[ht]
$$\xymatrix@R=5pt@C=4pt{
&&&&&&&&&& \\
&& \vdots && && \vdots && \vdots  \\
M_{n-1,3} \ar@{--}[rr] \ar@{.}[dd] \ar@{.}[uu] \ar[dr] && M_{0,4} && \cdots && M_{n-3,1} \ar@{--}[rr] \ar[dr] && M_{n-2,2} \ar@{--}[rr] \ar[dr] && M_{n-1,3} \ar@{.}[dd] \ar@{.}[uu] \\
& M_{0,3} \ar@{--}[l] \ar@{--}[r] \ar[ur] \ar[dr] && && && \ar@{--}[l] M_{n-2,1} \ar@{--}[rr] \ar[ur] \ar[dr] && M_{n-1,2} \ar@{--}[r] \ar[ur] \ar[dr] & \\
M_{0,2} \ar@{--}[rr] \ar[ur] && M_{1,3} && \cdots && M_{n-2,0} \ar@{--}[rr] \ar[ur] && M_{n-1,1} \ar@{--}[rr] \ar[ur] && M_{0,2}
}$$
\caption{The AR-quiver of $\T_n$}
\label{fig:ARquiverTn}
\end{figure}
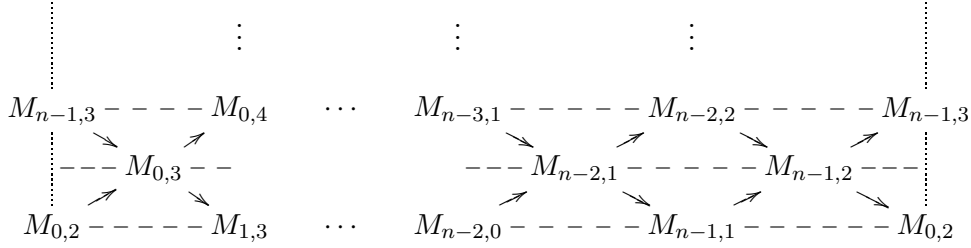

An additive subcategory of $\T$ is said to be of {\em finite type} if it contains only finitely many indecomposable objects.
Otherwise it is said to be of {\em infinite type}.
Some particular subcategories of $\T$ are important
for this paper.
For each fixed $i$ in $\{0, \dots n-1 \}$, we consider the additive subcategory
whose objects are the indecomposable objects $M_{i,i+t}$, for all $t>1$.
This is called a {\em ray}, and is denoted $\R_i$.
Dually, for each  $i$ in $\{0, \dots n-1\}$, we consider the
additive subcategory whose indecomposable objects are all indecomposable
objects $M_{i-u,i}$, for all $u>1$. This is called a {\em coray}, and
denoted $\C_i$.

For each $i$ in $\{0, \dots, n-1\}$, and each
$t>1$, the {\em wing} $\W_{i,i+t}$ is the
additive subcategory of $\T$ whose indecomposable objects are the
$M_{j,j+u}$ with $u\geq 2$, $i \leq j \leq i+t-2$ and
$j+u \leq i+t$. It contains a unique indecomposable object $M_{i,i+t}$
of maximal length. The objects $M_{i,i+2}$ and $M_{i+t-2,i+t}$ lie at the
bottom left and bottom right corners of the wing, respectively, in the
collection of vertices corresponding to the indecomposable objects of the wing in the
AR-quiver of $\T$.
For $t \leq 1$ we let $\W_{i,i+t}$ be the zero subcategory.
We also denote the wing of an indecomposable object $X$ by $\W_X$.

Due to the following well-known fact, we can apply
results from the previous section in our analysis of $\T$.

\begin{lemma}\label{wingsareA}
For $u \leq n+1$ the wing $\W_{i,i+u}$ in $\T_n$ is equivalent to
the module category $\module KQ$, where $Q$ is a linearly
oriented quiver of Dynkin type $A_{u-1}$. If $u\leq n$ the equivalence
is exact.
\end{lemma}

Note that if $u=n+1$, the object corresponding to the projective-injective
indecomposable module in $\module KQ$ under the above equivalence is not $\Ext$-projective in the wing.

We denote by $\Module \Lambda$ the category of all left
$\Lambda$-modules. Let $\varinjlim \T$ be the subcategory of $\Module \Lambda$
whose objects are direct limits of filtered direct systems of objects in $\T$.
Note that $\varinjlim \T$ contains the  {\em Pr\"{u}fer modules}, i.e.\ the modules $M_{i,\infty}$, $i=0,2,\ldots n-1$ obtained as direct limits
of the indecomposable objects in the rays, i.e.\
$$M_{i,\infty} = \varinjlim (M_{i,i+2} \to M_{i,i+3} \to M_{i,i+4} \to \cdots).$$
Let $\varprojlim \T$ be the subcategory of $\Module \Lambda$
whose objects are inverse limits of filtered inverse systems of objects in $\T$.
This category contains the {\em adic} modules, which are
obtained as inverse limits along a coray:
$$M_{-\infty,i} = \varprojlim (\cdots \to M_{i-4,i} \to M_{i-3,i} \to
M_{i-2,i} ).$$

Let $\overline{\T}$ be the subcategory of $\Module \Lambda$ whose
objects are all filtered direct limits or filtered inverse limits of
objects in $\T$. This category clearly contains $\varinjlim \T$
and $\varprojlim \T$.
We extend the definition of $\sigma$ to Pr\"{u}fer and adic modules
with the convention that $\sigma(\pm \infty)=\pm \infty$: note that any Pr\"{u}fer
module has a unique name $M_{i,\infty}$ if we take $i$ in $\{0,1,\ldots ,n-1\}$, and
similarly for adic modules.

Recall that a $\Lambda$-module $M$ is \emph{pure-injective} if
the canonical map $M \rightarrow DDM$ is a split monomorphism.
For background on pure-injective modules, and other definitions,
see e.g.~\cite{jensenlenzing89} or~\cite{prest09}.
It can be shown (see \cite{bk1}) that the category $\overline{\T}$
has the following properties:

\begin{itemize}
\item[$\bullet$] All objects are pure-injective as $\Lambda$-modules.
\item[$\bullet$] Any object is determined by its indecomposable direct summands.
\item[$\bullet$] The indecomposables in $\overline{\T}$ are exactly the indecomposables
  $M_{ij}$ in $\T$, the Pr{\"u}fer modules $M_{i,\infty}$ and the
  adic modules $M_{-\infty,i}$.
\end{itemize}

A module $M$ in $\overline{\T}$ is called {\em rigid} if
$\Ext^1_{\Lambda}(M,M) = 0$. Note that since all objects in
$\overline{\T}$ are pure-injective, this definition is equivalent to having
$\Ext^1_{\Lambda}(M',M'') = 0$ for all indecomposable direct summands $M',M''$ of
$M$. Now, a rigid module $M$ in $\overline{\T}$ is called {\em maximal
  rigid}, if $\Ext^1_{\Lambda}(M \amalg X,M \amalg X) = 0$ for an
indecomposable $X$ in $\overline{\T}$ implies
that $X$ is isomorphic to a direct summand in $M$.

\subsection{Geometric model}\label{cm}

We now give a geometric model for $\overline{\T}$. This extends the
known geometric model for $\T$~\cite{bama,warkentin}.
Consider an annulus $\AAA(n)$ with $n$ marked points on the outer boundary.
The points are labelled $0, 1, \dots, n-1$, and arranged anticlockwise
(see Figure~\ref{fig:annulus}).

\begin{figure}[ht]
\psfragscanon
\psfrag{0}{$\scriptstyle 0$}
\psfrag{1}{$\scriptstyle 1$}
\psfrag{2}{$\scriptstyle 2$}
\psfrag{n-1}{$\scriptstyle n-1$}
\includegraphics[width=2.5cm]{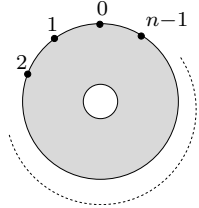}
\caption{An annulus with $n$ marked points on its outer boundary.}
\label{fig:annulus}
\end{figure}

Let $\UUU(n)$ denote the universal cover of $\AAA(n)$, with marked
points corresponding to $\mathbb{Z}$ (and with $0,1, \dots, n-1$
lying in a fundamental domain). See Figure~\ref{fig:universalcover}.

\begin{figure}[ht]
\psfragscanon
\psfrag{-1}{$\scriptstyle -1$}
\psfrag{0}{$\scriptstyle 0$}
\psfrag{1}{$\scriptstyle 1$}
\psfrag{2}{$\scriptstyle 2$}
\psfrag{n-1}{$\scriptstyle n-1$}
\psfrag{n}{$\scriptstyle n$}
\includegraphics[height=2cm]{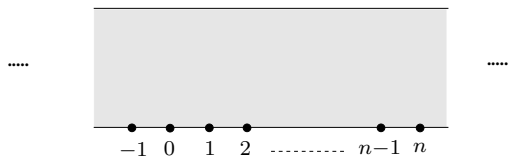}
\caption{The universal cover of the annulus in Figure~\ref{fig:annulus}.}
\label{fig:universalcover}
\end{figure}

For integers $i,j$ with $i+2 \leq j$, let $[i,j]$ denote the arc in
$\UUU(n)$ with starting point $i$ and ending point $j$, oriented from $i$ to $j$.
We also allow arcs which have only one end-point: arcs of the form $[i,\infty]$
(respectively, $[-\infty,j]$ which start
at $i$ (respectively, end at $j$) and are oriented in the positive $x$ direction.
See Figure~\ref{fig:infinitearc1}.

\begin{figure}
\psfragscanon
\psfrag{i}{$i$}
\psfrag{iinf}{$[i,\infty]$}
\psfrag{i-inf}{$[-\infty,i]$}
\includegraphics[width=12cm]{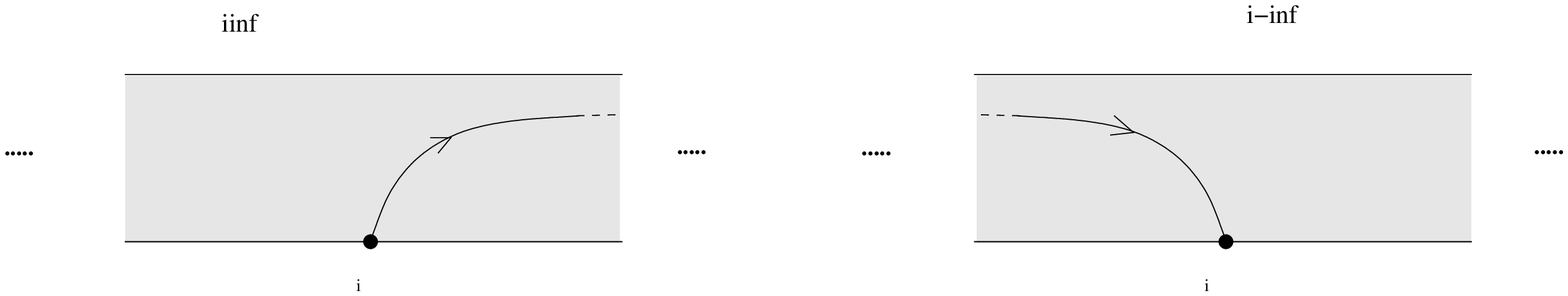}
\caption{Infinite arcs in $\UUU(n)$.}
\label{fig:infinitearc1}
\end{figure}

Let $\pi_n([i,j])$ denote the corresponding arc in $\AAA(n)$ and
let $\widetilde{A}=\widetilde{\A}(\AAA(n))$ denote the set of (isotopy classes
of) such arcs. It contains the set $\A=\A(\AAA(n))$ of arcs of the form $\pi_n([i,j])$
with $i,j$ finite. The map $\psi:\widetilde{\A}\rightarrow \ind \overline{\T}$
sending $\pi_n([i,j])$ to $M_{ij}$ is a bijection.
The infinite arcs are displayed in Figure~\ref{fig:infinitearc2}.

\begin{figure}
\psfragscanon
\psfrag{i}{$i$}
\psfrag{iinf}{$\pi_n([i,\infty])$}
\psfrag{i-inf}{$\pi_n([-\infty,i])$}
\includegraphics[height=3cm]{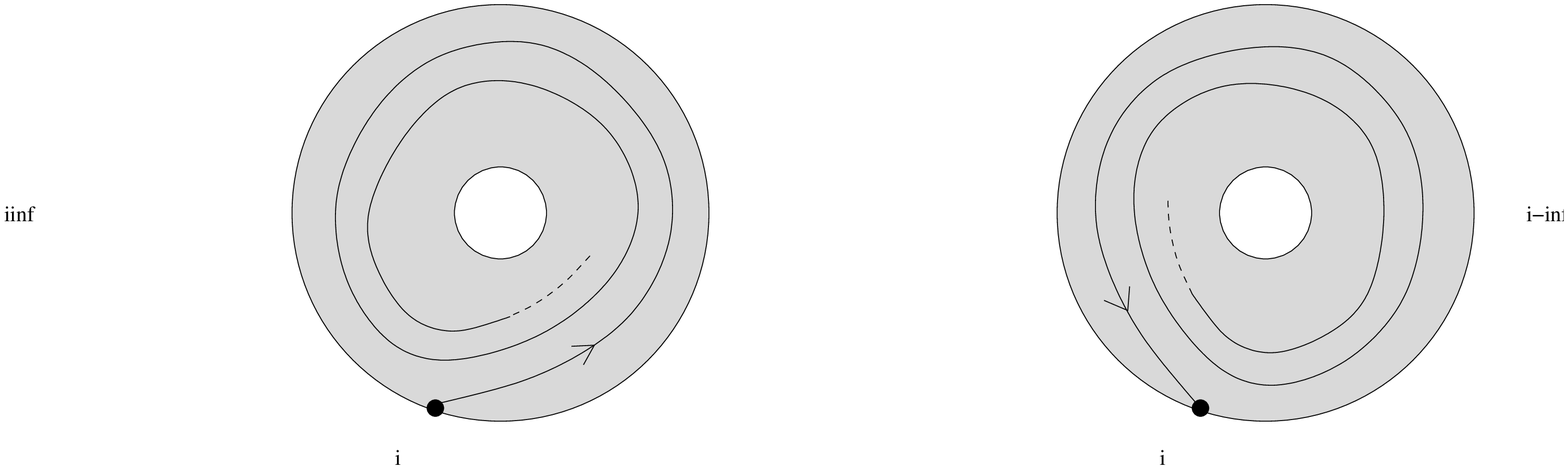}
\caption{Infinite arcs in $\AAA(n)$.}
\label{fig:infinitearc2}
\end{figure}

Define a quiver with vertices given by the elements in $\A$ and arrows:
$$\pi_n([i,j]) \to \pi_n([i,j+1])$$ and $$\pi_n([i,j]) \to
\pi_n([i+1,j]) \text{ (if $j \neq i+2)$} $$
Defining a translate using the formula $\tau(\pi_n([i,j])) = \pi_n([i-1,j-1])$,
this becomes a translation quiver. We call this the \emph{(translation) quiver
of $\A(\AAA(n))$}.

By~\cite[Lemma 2.5]{bama},~\cite[4.18]{warkentin} (or, using unoriented arcs,~\cite[\S3.4]{bz},~\cite{gehrig}), we have:

\begin{proposition}
The restriction of $\psi$ to $\A$ gives an isomorphism between
the translation quiver of $\A(\AAA(n))$ and the AR-quiver of $\T_n$.
\end{proposition}

Note that the convention that $M_{\sigma^k(i),\sigma^k(j)}=M_{ij}$ corresponds exactly
to the fact $\pi_n([\sigma^k(i),\sigma^k(j)])=\pi_n([i,j])$ for any integer $k$.

For arcs $\alpha, \beta$ in $\A(\AAA(n))$, let
$I(\alpha, \beta)$ be the minimum number of intersections between arcs in the isotopy classes $\alpha$ and
$\beta$, not allowing non-transverse or multiple intersections. Similarly we let $I^+(\alpha, \beta)$ (resp. $I^-(\alpha, \beta)$) denote the number of
positive (resp. negative) crossings between $\alpha$ and $\beta$ (see Figure~\ref{fig:neg} for an example of a negative crossing).
We will now prove the following result:

\begin{theorem}\label{exts}
Given indecomposable objects $M_{ij}$ and $M_{i'j'}$ in
$\overline{\T}$. Then:
$$\Ext^1(M_{ij},M_{i'j'}) \cong \prod_{I^-(\pi_n([i,j]),
  \pi_n([i',j']))} K .$$
\end{theorem}

In the case where $i,i',j,j'$ are all finite, the result is proved in
\cite[Thm.\ 3.7]{bama},~\cite[Thm. 4.23]{warkentin}.
See also \cite{bz} for further results in this direction.
We first recall some results we will need:

\begin{lemma}\label{limits}
Let $X,Y$ be arbitrary $\Lambda$-modules,
$(X_j)_j$ an arbitrary filtered direct system of modules and
$(Y_j)_j$ an arbitrary filtered inverse system of modules.
\begin{itemize}
\item[(a)] $\Hom(\varinjlim X_j, Y) \simeq \varprojlim \Hom(X_j,Y)$.
\item[(b)] If $X$ is finitely generated, then $\Hom(X,\varinjlim Y_j)\simeq \varinjlim \Hom(X,Y_j)$
\item[(c)] If the $Y_j$ are finitely generated, then
$\varprojlim Y_j\simeq D\varinjlim DY_j$.
\item[(d)] If $Y$ is pure-injective, then
$\Ext^1(\varinjlim X_j, Y) \simeq \varprojlim \Ext^1(X_j,Y)$.
\item[(e)] If the $Y_j$ are finitely generated, then
$\Ext^1(X, \varprojlim Y_j) \simeq \varprojlim \Ext^1(X,Y_j)$.
\end{itemize}
\end{lemma}

\begin{proof} For (a) see, for example,~\cite{trlifaj03}.
For (b), see~\cite[Lemma 1.6]{krausesolberg03}
or~\cite[Sect. 1.5]{crawleyboevey98}.
For (c), we have, using part (a):
\begin{align*}
D\varinjlim DY_j &= \Hom(\varinjlim \Hom(Y_j,K),K) \\
&\simeq \varprojlim \Hom(\Hom(Y_j,K),K)
\simeq \varprojlim DDY_j \simeq \varprojlim Y_j,
\end{align*}
as required.
Part (d) is proved in~\cite[Prop.\ I.10.1]{auslander78}.
For (e), we recall that $\Ext^1(X,DY)\simeq \Ext^1(Y,DX)$ for all modules
$X$ and $Y$. Using parts (c) and (d)
and the fact~\cite[Prop.\ 4.3.29]{prest09} that $DX$ is pure-injective for
any module $X$, we have:
\begin{align*}
\Ext^1(X,\varprojlim Y_j) &\simeq \Ext^1(X,D(\varinjlim DY_j))
\simeq \Ext^1(\varinjlim DY_j,DX) \\
&\simeq \varprojlim \Ext^1(DY_j,DX)
\simeq \varprojlim \Ext^1(X,DDY_j) \\
&\simeq \varprojlim \Ext^1(X,Y_j),
\end{align*}
and (e) is shown.
\end{proof}

We also need the following (see e.g.~\cite[Sect.\ 3.1]{crawleyboevey98}).

\begin{lemma} \label{ARformula}
For modules $X$ and $Y$ with $X$ finitely generated, we have
$D\Ext^1(X,Y)\cong \Hom(Y,\tau X)$ and $\Ext^1(Y,X)\cong D\Hom(\tau^{-1}X,Y)$.
\end{lemma}

Note that if $X,Y$ are finitely generated, then the first formula can also
be written $\Ext^1(X,Y)\cong D\Hom(Y,\tau X)$. We recall the following
(see, for example,~\cite[p46]{ringel00}).

\begin{lemma} \label{homvanishing}
\begin{itemize}
\item[(a)] If $X$ is a Pr\"{u}fer module and $Y$ is a finitely
generated module then $\Hom(X,Y)=0$.
\item[(b)] If $X$ is a finitely generated module and $Y$ is an adic
module then $\Hom(X,Y)=0$.
\end{itemize}
\end{lemma}

With arguments as in \cite{bama}, the crossing numbers can now be computed as follows.
Recall that $\sigma \colon \mathbb{Z} \to  \mathbb{Z}$ is the function $i \mapsto i+n$.

\begin{proposition}\label{crossingnumbers}
We have the following:
\begin{itemize}
\item[(a)]  $I^-(\pi_n([i,\infty]), \pi_n([a,b])) = \mid \{m \in
  \mathbb{Z}  \colon a < \sigma^m(i)< b \} \mid$;
\item[(b)] $I^-(\pi_n([a,b]), \pi_n([i,\infty])) = I^+(\pi_n([i,\infty]), \pi_n([a,b])) = 0$;
\item[(c)] $I^-(\pi_n([-\infty,j]), \pi_n([a,b]))= 0$;
\item[(d)] $I^- (\pi_n([a,b]),\pi_n([-\infty,j])) =
I^+(\pi_n([-\infty,j]), \pi_n([a,b]))
= \mid \{m \in  \mathbb{Z}  \colon a < \sigma^m(i)< b \} \mid$;
\item[(e)] $I^-(\pi_n([i,\infty]), \pi_n([-\infty, i'])) = \aleph_0$,
  for all $i,i'$ in $\{0,\dots, n-1\}$;
\item[(f)] $ I^-(\pi_n([i,\infty]),
    \pi_n([i',\infty])) = 0$ for all $i,i'$ in $\{0,\dots, n-1\}$;
\item[(g)] $I^-(\pi_n([-\infty, i]), \pi_n([i',\infty])) =  0$ for all $i,i'$ in $\{0,\dots, n-1\}$;
\item[(h)] $I^-(\pi_n([-\infty,
i]), \pi_n([-\infty, i'])) = 0$ for all  $i,i'$ in $\{0,\dots, n-1\}$.
\end{itemize}
\end{proposition}

\begin{proof}[Proof of Theorem \ref{exts}]
We need to compute $\Ext^1(X,Y)$ for all pairs of indecomposables
$X,Y$ in $\overline{\T}$, and compare these with the crossing-numbers
from Proposition \ref{crossingnumbers}.
We first determine $\Ext^1(M_{i,\infty}, M_{ab})$. By Lemma~\ref{ARformula},
$$\Ext^1(M_{i,\infty},M_{ab})\cong D\Hom(M_{a+1,b+1},M_{i,\infty}).$$
Any morphism $f:M_{a+1,b+1}\rightarrow M_{i,\infty}$ must factor through the
unique submodule of $M_{i,\infty}$ of length $b-a-1$, which is isomorphic to
$M_{i,i+b-a}$.
Hence $\Hom(M_{a+1,b+1},M_{i,\infty})\cong \Hom(M_{a+1,b+1},M_{i,i+b-a})$.
Thus $\Ext^1(M_{i,\infty},M_{ab})$ equals the dimension of this last space,
i.e.\ the number of times the simple top $M_{b-1,b+1}$ of $M_{a+1,b+1}$
appears as a composition factor in $M_{i,i+b-a}$, which, using arguments
as in~\cite{bama}, is given by
$$|\{n\in\mathbb{Z}\,:\,a<\sigma^n(i)<b\}|.$$

Using a dual argument, we obtain:
$$\dim \Ext^1(M_{i,-\infty}, M_{ab}) =
\mid \{n \in \mathbb{Z} \colon a < \sigma^n(i) <b \} \mid.$$

By Lemma~\ref{ARformula},
$$D\Ext^1(M_{ab}, M_{i, \infty}) \simeq \Hom(M_{i,\infty},M_{a-1, b-1}).$$
We see that $\Ext^1(M_{ab},M_{i,\infty})=0$ by Lemma~\ref{homvanishing}(a).

Similarly, by Lemma~\ref{ARformula}, we have that
$$\Ext^1(M_{-\infty,i}, M_{ab}) \simeq D\Hom(\tau^{-1}M_{ab},M_{-\infty,i})=0,$$
using Lemma~\ref{homvanishing}(b).

We next compute $\Ext^1(M_{i,\infty}, M_{-\infty,i'})$. Using Lemma
\ref{limits}(e), we have that
$$\Ext^1(M_{i,\infty}, M_{-\infty,i'}) = \Ext^1(M_{i,\infty},
\varprojlim M_{j,i'}) \simeq \varprojlim \Ext^1(M_{i,\infty}, M_{j,i'}),$$
where $j$ is the dummy variable in the limit.
The maps
$$\Ext^1(M_{i,\infty}, M_{j,i'}) \to \Ext^1(M_{i,\infty},
M_{j-1,i'})$$
are surjective. As $j$ tends to $-\infty$, the dimension of
$\Ext^1(M_{i,\infty}, M_{j,i'})$ is unbounded (by the above formula for it).
Therefore, this limit evaluates to $\prod_{\aleph_0} K$.

It is shown in \cite[Lemma 2.7]{bk1}, that, for all $i,i'$
in $\{0,\dots ,n-1\}$, all of \linebreak
$\Ext^1(M_{-\infty, i}, M_{i',\infty})$, 
$\Ext^1(M_{i,\infty},M_{i',\infty})$ and $\Ext^1(M_{-\infty, i}, M_{-\infty,i'})$
vanish.

Comparing the crossing numbers with the dimensions of the corresponding
Ext-groups concludes the proof of the theorem.
\end{proof}

\subsection{Reflection}

Recall (see \cite{bk1}) that there is a map  $M \mapsto M^{\vee}$, which gives a bijection
on the indecomposable objects in $\overline {\T}$. With our notation,
the map is given by $M_{ij} \mapsto M_{-j,-i}$ for $i,j$ in
$\mathbb{Z} \cup \{ \pm \infty \}$.

For a subcategory $\X$ of $\overline{\T}$, we let $\X^{\vee}$ denote the
subcategory $\{X^{\vee} \mid X \in \X \}$. The map has the following properties.
\begin{lemma}\label{properties}
\begin{itemize}
\item[(a)] $\dim \Hom(X,Y) = \dim \Hom(Y^{\vee}, X^{\vee})$ for
  $X,Y$ in $\overline{\T}$;
\item[(b)] $\Ext^1(X,Y) = 0$ if and only if $\dim \Ext^1(Y^{\vee},
  X^{\vee}) = 0$ for
  $X,Y$ in $\overline{\T}$;
\item[(c)]$(\tau X)^{\vee} = \tau^{-1} X^{\vee}$ for all
  indecomposables $X$ in $\T$.
\item[(d)] For each object $X$ in $\overline \T$, we have
$(X^{\hperp})^{\vee} = {^{\hperp}(X^{\vee})}$ and
$(X^{\vee})^{\eperp} = (^{\eperp}X)^{\vee}$.
\end{itemize}
\end{lemma}

We will also need the following.

\begin{lemma} \label{l:twistgen}
For a maximal rigid object $\rigid$ in $\overline{\T}$, we have
\begin{itemize}
\item[(a)] $(\Gen {\rigid}\cap \T)^{\vee} = \Cogen({\rigid}^{\vee})\cap \T$;
\item[(b)] $(\Cogen {\rigid}\cap \T)^{\vee} = \Gen({\rigid}^{\vee})\cap \T$.
\end{itemize}
\end{lemma}

\begin{proof}
This follows directly from the definition of the map using
Lemma~\ref{homvanishing} and the fact that
an indecomposable $X$ in $\T$ is generated by a maximal
rigid object ${\rigid}$ if and only if it is generated by an indecomposable direct
summand in ${\rigid}$ (see also \cite{bk2}).
\end{proof}

\section{Torsion pairs and maximal rigid objects}

In this section we give an improvement of a result from
\cite{bk1}, and discuss a combinatorial interpretation.
The idea is to link maximal rigid objects in $\overline{\T}$ with
torsion pairs in $\T$.

\subsection{A bijection}
A rigid object in $\overline{\T}$ is said to be of {\em Pr\"{u}fer type}
if it has a Pr\"{u}fer module as a direct summand, and it is said to be
of {\em adic type} if it has an adic module as a direct summand.

An object ${\rigid}$ in $\varinjlim \T$ is maximal rigid if and only if it is
cotilting in the category $\varinjlim \T$, in the sense of Colpi~\cite{colpi99}.
This is proved in~\cite{bk2} (see~\cite[1.10]{bk2} and note that,
by~\cite[Lemma 1.2]{bk2}, a finitely presented $\tilde{\Lambda}_n$-module
is maximal rigid if and only if it is a tilting module, where
$\tilde{\Lambda}_n$ denotes the completion of the path algebra of an
oriented $n$-cycle).

We consider equivalence classes of maximal rigid
objects in $\overline{\T}$, where two maximal rigid objects are considered to
be equivalent if they have the same indecomposable direct summands.
The following is proved in~\cite{bk1}.

\begin{proposition}\label{propbk1}
Let ${\rigid}$ be a rigid object in $\overline{\T}$.
\begin{itemize}
\item[(a)] ${\rigid}$ is maximal rigid in $\overline{\T}$ if and only if it has $n$ pairwise nonisomorphic indecomposable direct summands.
\item[(b)] If ${\rigid}$ is maximal rigid, it is either of Pr{\"u}fer type or of adic type, but not both.
\item[(c)] A rigid object of Pr{\"u}fer type lies in the subcategory
$\varinjlim \T \subset \overline{\T}$.
\end{itemize}
\end{proposition}

We will give a parallel result concerning torsion pairs in $\T$.
A torsion pair $(\Torsion,\Free)$ in $\T$ is said to be of {\em ray type} if
$\Free$ contains at least one ray of $\T$.  It is said to be of {\em coray type} if
$\Torsion$ contains at least one coray of $\T$.
A class of objects (or subcategory) $\X$ of a category $\C$ is said to be {\em generating},
if for each map $f$ in $\C$, there is an object $X$ in $\X$ with
$\Hom(X,f) \neq 0$. Dually, one can define {\em cogenerating} classes
of objects, i.e.\ $\X$ is cogenerating if, for each map $f$ in $\C$,
there is an object $X$ in $\X$ with $\Hom(f,X)\neq 0$.

\begin{lemma}\label{onefinite}
Let $(\Torsion,\Free)$ be a torsion pair in $\T$. Then it is not the case that
both $\Torsion$ and $\Free$ are of finite type.
\end{lemma}

\begin{proof}
Suppose that $\Torsion$ and $\Free$ are both of finite type.
By the definition of a torsion pair, every object $M$ in $\T$
must be the middle term of an exact sequence
\begin{equation*}
0 \to \torsion \to M \to \free \to 0,
\end{equation*}
where $\torsion\in \Torsion$ and $\free\in \Free$.
Since $\T$ is serial, if $M$ is indecomposable then so are
$\torsion$ and $\free$. It follows that $\T$ itself is of finite
type, a contradiction.
Hence, we cannot have that both $\Torsion$ and $\Free$ are of finite type.
\end{proof}

We next prove that the torsion pairs with $\Torsion$ of infinite type
are exactly those of coray type. For an indecomposable object $X$ in $\T$
we write $\C_X$ for the coray containing $X$ and $\R_X$ for the ray containing
$X$.

\begin{lemma}\label{coray-type}
Let $(\Torsion,\Free)$ be a torsion pair in $\T$, where $\T$ has rank $n$.
Assume $\Torsion$ is of infinite type. Then the following hold:
\begin{itemize}
\item[(i)] $\Torsion$ contains an indecomposable object $X$, with $l(X) = n$.
\item[(ii)] $\Free$ is of finite type.
\item[(iii)] $\Torsion$ contains the coray $\C_X$.
\item[(iv)] $\Torsion$ cogenerates $\T$.
\end{itemize}
\end{lemma}

\begin{proof}
First note that since $\Torsion$ is of infinite type, there is no limit on the length of the indecomposable objects
in $\Torsion$. Since $\Torsion$ is closed under factor objects, it must therefore contain an indecomposable object with $l(X) = n$.
Hence (i) holds.
Since $\Torsion$ is closed under factor objects, the
indecomposable objects in the coray $\C_X$ below $X$ are contained in
$\Torsion$; more precisely $\ind(\C_X) \cap \{Y \mid l(Y) \leq n \} \subset \Torsion$.
Assume first that $n>1$.
Let $X'$ be the (uniquely defined) indecomposable object in $\T$ such that there is an irreducible
monomorphism $X' \to X$.
Since $l(X) = n$, we have that $\Hom(X,Y) = 0$ for an indecomposable object
$Y$ in $\T$ if and only if $Y$ is in the wing
$\W_{X'}$. By the definition of a torsion pair, we have that
$\Free$ is contained in $\W_{X'}$ so is of finite type.
If $n=1$, we have $\Hom(X,Y)\neq 0$ for any indecomposable
object $Y$ in $\T$, so $\Free$ is the zero subcategory and (ii) holds.
By Lemma~\ref{l:tpclosure}, (iii) holds, while (iv) is a direct consequence of (iii).
 \end{proof}

We state the dual version of Lemma \ref{coray-type}.

\begin{lemma}\label{ray-type}
Let $(\Torsion,\Free)$ be a torsion pair in $\T$, where $\T$ has rank $n$.
Assume $\Free$ is of infinite type. Then the following hold:
\begin{itemize}
\item[(i)] $\Free$ contains an indecomposable object $X$, with $l(X) = n$.
\item[(ii)] $\Torsion$ is of finite type.
\item[(iii)] $\Free$ contains the ray $\R_X$.
\item[(iv)] $\Free$ generates $\T$.
\end{itemize}
\end{lemma}

Combining Lemma \ref{onefinite}, \ref{coray-type} and  \ref{ray-type}, we obtain the following direct consequence.

\begin{corollary}\label{cortorsion}
Let $(\Torsion,\Free)$ be a torsion pair in $\T$.
\begin{itemize}
\item [(a)]  The following are equivalent
\begin{itemize}
\item[(i)] $\Torsion$ is of infinite type;
\item[(ii)] $\Torsion$ contains a coray;
\item[(iii)] $\Torsion$ cogenerates $\T$;
\item[(iv)] $\Free$ is of finite type.
\end{itemize}
\item[(b)]
$({\Torsion, \Free})$ is either of ray or of coray type (and not both).
\end{itemize}
\end{corollary}

Moreover, by Lemma~\ref{l:tpclosure} and a direct application of Lemma \ref{properties},
we obtain the following.

\begin{lemma}
The map $M \to M^{\vee}$,  maps a torsion pair $(\Torsion, \Free)$ to a
torsion pair $(\Free^{\vee}, \Torsion^{\vee})$. Moreover if  $(\Torsion, \Free)$ is of
ray-type, then  $(\Free^{\vee}, \Torsion^{\vee})$ is of coray-type and vice-versa.
\end{lemma}

For a maximal rigid object ${\rigid}$ of Pr{\"u}fer type, consider the
subcategory $\Free_{\rigid} = {^{\eperp}{\rigid}} \cap \T$. We set $\Torsion_{\rigid} =
{^{\hperp}(\Free_{\rigid})} \cap \T.$
For a maximal rigid object ${\rigid}$ of
adic type, we define
$\Torsion_{\rigid} ={\rigid} ^{\eperp} \cap \T$ and $\Free_{\rigid} =
(\Torsion_{\rigid})^{\hperp} \cap \T$.
We have the following reformulation of a result of~\cite{bk2}:

\begin{theorem}\label{propbk2}
The map ${\rigid} \mapsto (\Torsion_{\rigid}, \Free_{\rigid})$ gives
a one-to-one correspondence between equivalence classes of
maximal rigid objects in $\varinjlim \T$ and
torsion pairs in $ \T$ with the property that $\Free$ generates $\T$.
\end{theorem}

Now, using Lemma \ref{properties}, we obtain a commutative square (*):
$$
\xymatrix{
\{\text{maximal rigid objects of Pr{\"u}fer type in } \overline{\T} \}
\ar[r]  \ar[d] & \{\text{torsion pairs of ray-type in } \T \} \ar[d] \\
\{\text{maximal rigid objects of adic type in } \overline{\T} \}
\ar[r]  & \{\text{torsion pairs of coray-type in } \T  \}
}
$$
where the horizontal maps are given by ${\rigid} \mapsto (\Torsion_{\rigid}, \Free_{\rigid})$ and
the vertical maps are induced by $M \mapsto M^{\vee}$.

As a direct consequence of Lemma~\ref{properties} (as
also observed in \cite{bk1}), we have that the left vertical map in (*)
is a bijection. We have already observed that the right vertical map
is a bijection. The upper horizontal map is bijective by Theorem~\ref{propbk2},
combining with Proposition~\ref{propbk1} and Lemma~\ref{ray-type},
and it follows that the lower horizontal map is bijective.

Combining the commutative diagram of bijections (*) with Corollary \ref{cortorsion}
we obtain the following improvement of Theorem \ref{propbk2}.

\begin{theorem} \label{t:mainbijection}
The map ${\rigid} \mapsto (\Torsion_{\rigid}, \Free_{\rigid})$ gives a bijection between
\begin{itemize}
\item[$\bullet$] Equivalence classes of maximal rigid objects in $\overline{\T}$
\item[$\bullet$] Torsion pairs $(\Torsion,\Free)$ in $\T$, and
\end{itemize}
\end{theorem}

\begin{corollary}
The number of torsion pairs in $\T$ is $2\binom{2n-1}{n-1}$.
\end{corollary}

\begin{proof}
By~\cite[2.4, 5.2 and B.1]{bk2},
the number of cotilting objects in $\varinjlim \T$
is $\binom{2n-1}{n-1}$. The result then follows from Proposition~\ref{propbk1},
the above diagram of maps, and Theorem~\ref{t:mainbijection}.
\end{proof}

Note that the above results hold in the case $n=1$: in this case there are two maximal
rigid objects: the unique Pr\"{u}fer and adic modules, and two corresponding torsion pairs,
$(0,\T)$ and $(\T,0)$ respectively.

\subsection{Alternative and explicit descriptions of $\Torsion_{\rigid}$ and $\Free_{\rigid}$}

We give alternative and more explicit descriptions for the subcategories
$\Torsion_{\rigid}$ and $\Free_{\rigid}$ corresponding to a maximal rigid object ${\rigid}$ in
$\overline{\T}$.

The following observation is useful.

\begin{lemma}\label{perps}
Let $\T$ be a tube of rank $n$. Then we have:
\begin{itemize}
\item[(a)]$^{\eperp} M_{i,\infty} \cap \T = \T$;
\item[(b)]$M_{i,\infty}^{\eperp} \cap \T = \W_{i,i+n}$
\end{itemize}
\end{lemma}

We now give an explicit description of the torsion pair $(\Torsion_{\rigid},\Free_{\rigid})$
corresponding to a  maximal rigid module ${\rigid}$. We first of
all give a combinatorial lemma concerning wings, which is easy to
check.

\begin{lemma} \label{l:wings}
Let $0 \leq i_0 < i_1 < \cdots < i_{k-1} \leq n-1$ be integers.
Then
\begin{itemize}
\item[(i)] We have:
$$\bigcap_{r=0}^{k-1} \W_{i_r,i_r + n} = \coprod_{r=0}^{k-1} \W_{i_r,i_{r+1} }$$
(where the subscripts are interpreted modulo $k$ and, for $r=k-1$, we
interpret $i_{r+1}=i_0$ as $i_0+n$).
\item[(ii)] The wing $\W_{i_r,i_{r+1} }$ is zero if and only if
$i_{r+1} - i_{r}= 1$.
\item[(iii)] The wings $\W_{i_r,i_{r+1}}$ do not overlap,
i.e.\ each indecomposable object in $\T$ belongs to at most one $\W_{i_r,i_{r+1}}$.
\item[(iv)] If $X$ and $Y$ lie in different wings among the
  $\W_{i_r,i_{r+1} } $, then $\Ext^1(X,Y) = 0 = \Ext^1(Y,X)$.
\end{itemize}
\end{lemma}

For an example illustrating Lemma~\ref{l:wings}(i),
with $n=10$, $k=4$, $i_0=0$, $i_1=4$, $i_2=7$ and $i_3=8$, see
Figure~\ref{fig:wingsexample}.

\begin{figure}[hb]
$$\xymatrix@!0@=0.5cm{
&&&&&&&&&& &&&&&&&&&& \\
&& \scriptstyle M_{7,17} && \scriptstyle M_{8,18} &&&& \scriptstyle M_{0,10} && &&&&&& \scriptstyle M_{4,14} &&&& \\
\circ && *=0{\circ} \ar@{-}[rrrrrrrrdddddddd]+0 && *=0{\circ} \ar@{-}[rrrrrrrrdddddddd]+0 && \circ && *=0{\circ} \ar@{-}[rrrrrrrrdddddddd]+0 && \circ && \circ && \circ && *=0{\circ} \ar@{-}[rrrrdddd]+0 && \circ && \circ \\
& \circ && \circ && \circ && \circ && \circ && \circ && \circ && \circ && \circ && \circ \\
*=0{\circ} \ar@{-}[rruu]+0 && \circ && \circ && \circ && \circ && \circ && \circ && \circ && \circ && \circ && \circ \\
& \circ && \circ && \circ && \circ && \circ && \circ && \circ && \circ && \circ && \circ \\
*=0{\circ} \ar@{-}[rrrruuuu]+0 \ar@{-}[rrrrdddd]+0 && \circ && \circ && \circ && \circ && \circ && \circ && \circ && \circ && \circ && \circ \\
& \circ && \circ && \circ && \circ && \circ && \circ && \circ && \circ && \circ && \circ && \\
\circ && \bullet & \scriptstyle M_{0,4} & \circ && \circ && \circ && \circ && \circ && \circ && \circ && \circ && \circ \\
& \bullet && \bullet && \circ && \circ && \bullet & \scriptstyle M_{4,7} & \circ && \circ && \circ && \circ && \circ && \\
*=0{\bullet} \ar@{.}[uuuuuuuuuu]+0 \ar@{-}[rrrrrrrruuuuuuuu]+0 && \bullet && \bullet && \circ && *=0{\bullet} \ar@{-}[rrrrrrrruuuuuuuu]+0 && \bullet && \circ && *=0{\circ} \ar@{-}[rrrrrruuuuuu]+0 &&
*=0{\bullet} \ar@{-}[rrrruuuu]+0 && \circ && *=0{\bullet} \ar@{.}[uuuuuuuuuu]+0 \\
\scriptstyle M_{0,2} && \scriptstyle M_{1,3} && \scriptstyle M_{2,4} && \scriptstyle M_{3,5} && \scriptstyle M_{4,6} && \scriptstyle M_{5,7} && \scriptstyle M_{6,8} && \scriptstyle M_{7,9} && \scriptstyle M_{8,10} && \scriptstyle M_{9,11} && \scriptstyle M_{0,2}
}$$
\caption{The indecomposable objects in the intersection of the four wings
$\W_{0,10},\W_{4,14},\W_{7,17}$ and $\W_{8,18}$ in a tube of rank
$10$ corresponding to a maximal rigid object of Pr\"{u}fer
type. The objects lying in the intersection are indicated by filled-in circles.
The intersection coincides with $\W_{0,4}\amalg \W_{4,7}\amalg \W_{8,10}$
(note that the wing $\W_{i_3,i_4}=\W_{7,8}$ is zero).}
\label{fig:wingsexample}
\end{figure}
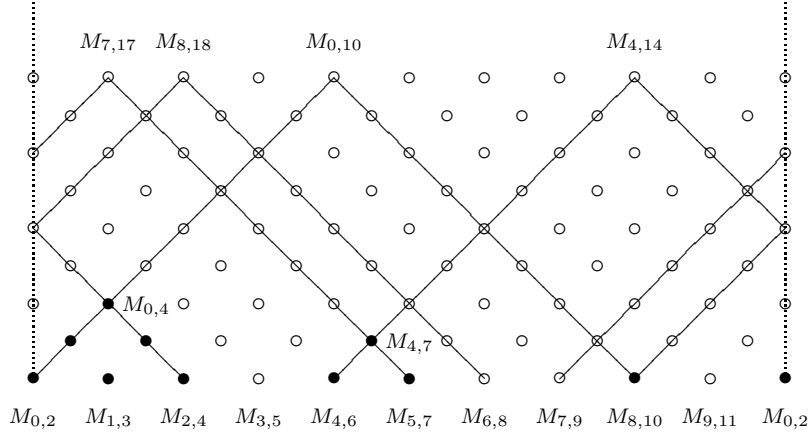

In the following proposition and the sequel, we adopt the same convention
for the wings as in Lemma~\ref{l:wings}(i).

\begin{proposition} \label{p:ftgt}
Let ${\rigid}$ be maximal rigid in $\overline{\T}$.
\begin{itemize}
\item[(a)] If ${\rigid}$ is of Pr{\"u}fer type, then
$$
(\Torsion_{\rigid}, \Free_{\rigid})  = ( {^{\hperp}{\rigid}} \cap \T, {^{\eperp}{\rigid}} \cap \T)
  = (\tau^{-1} (\Gen {\rigid} \cap \T), \Cogen {\rigid} \cap \T).
$$
\item[(b)]
Assume ${\rigid}$ is of Pr{\"u}fer type with Pr{\"u}fer summands
$M_{i_r,\infty}$
for $r=0, \dots, k-1$ where $0 \leq i_0 < i_1 < \cdots < i_{k-1}
\leq n-1$. Then $$(\Torsion_{\rigid}, \Free_{\rigid}) = (\coprod_{r=0}^{k-1}
  \Torsion_r, (\coprod_{r=0}^{k-1}\Free_r) \amalg \Free_{\infty})),$$ where
$(\Torsion_r, \Free_r)$ is a torsion pair in $\W_{i_r, i_{r+1}+1}$ with
$\Free_r$ containing all of the projective objects in $\W_{i_r,i_{r+1}+1}$
and $\Free_{\infty} = \coprod_{r=0}^{k-1} \R_{i_r}$.

\item[(c)]  If ${\rigid}$ is of adic type, then
$$(\Torsion_{\rigid}, \Free_{\rigid})  = ({\rigid}
  ^{\eperp} \cap \T, {\rigid} ^{\hperp} \cap \T) =
  (\Gen {\rigid} \cap \T, \tau(\Cogen {\rigid} \cap \T)).$$
\item[(d)] Assume ${\rigid}$ is of adic type with adic summands $M_{i_r,\infty}$
  for $r=0, \dots, k-1$,  where $0 \leq i_0 < i_1 < \cdots < i_{k-1}
\leq n-1$. Then $$(\Torsion_{\rigid}, \Free_{\rigid}) = ((\coprod_{r=0}^{k-1}
  \Torsion_r) \amalg \Torsion_{\infty}, \coprod_{r=0}^{k-1}\Free_r),$$ where
$(\Torsion_r, \Free_r)$ is a torsion pair in $\W_{i_r, i_{r+1}+1}$ with
$\Torsion_r$ containing all of the injective objects in $\W_{i_r,i_{r+1}+1}$ and
$\Torsion_{\infty} = \coprod_{r=0}^{k-1} \C_{i_{r+1}+1}$.
\end{itemize}
\end{proposition}

\begin{proof}
We give the details for (a) and (b), while statements (c) and (d) can
be proved similarly (or using Lemmas~\ref{properties} and ~\ref{l:twistgen}).
Let ${\rigid}$ be maximal rigid of Pr{\"u}fer type in $\overline{\T}$.
Our aim is to compute $\Free_{\rigid}={}^{\eperp} {\rigid}\cap \T$ and then
$\Torsion_{\rigid}={}^{\hperp}\Free_{\rigid}\cap \T$. We use the Pr\"{u}fer direct summands
of ${\rigid}$ in order to compute $\Free_{\rigid}$ more precisely in terms of a set
of wings in $\T$. We then use this to compute $\Torsion_{\rigid}$ using the theory of
torsion pairs in type A (see Section~\ref{s:Dynkin}).

Let ${\rigid}_{\T}$ be the direct sum of all indecomposable
direct summands of ${\rigid}$ which are finitely generated.
Let $M_{i_0,\infty},\ldots ,M_{i_{k-1},\infty}$
with $0\leq i_0<i_1<\cdots <i_{k-1}\leq n-1$ be the
indecomposable Pr{\"u}fer summands of ${\rigid}$.

By Lemma \ref{perps}, we then have that the finite part ${\rigid}_{\T}$
lies in $\cap_{r=0}^{k-1} \W_{i_r,i_r + n}$, which coincides with
$\coprod_{r=0}^{k-1} \W_{i_r,i_{r+1}}$ by Lemma~\ref{l:wings}(i).
We draw attention to the fact that the wings
in the statement of the proposition are slightly larger than these ---
this will become clearer later in the proof.

Hence, we consider the decomposition ${\rigid}_{\T} =\amalg_{r=0}^{k-1} {\rigid}_r$, where each ${\rigid}_r$ is
rigid in $\W_{i_r,i_{r+1}}$ (note that some of the ${\rigid}_r$ then
might be $0$, i.e.\ in the case when $i_{r+1}-i_r=1$).
Then it follows easily from Lemma~\ref{l:wings}(iv) that ${\rigid}_r$ is
maximal rigid in $\W_{i_r,i_{r+1}}$.
Since ${\rigid}_r$ is maximal rigid, its restriction to the abelian
category $\W_{i_r,i_{r+1}}$ is also cotilting. Our aim is to use this
decomposition of ${\rigid}$ to compute $\Free_{\rigid}$.

If $\W_{i_r,i_{r+1}}$ is nontrivial, $M_{i_r,i_{r+1}}$
is necessarily an indecomposable direct summand of ${\rigid}_r$.
It is then straightforward to check that
$$^{\eperp}(\amalg_{r=0}^{k-1} M_{i_r,i_{r+1}} \cap \T) =
(\coprod_{r=0}^{k-1}\W_{i_r,i_{r+1}}) \amalg \coprod_{r=0}^{k-1} \R_{i_r},$$
and from this it follows that $\Free \subset (\coprod_{r=0}^{k-1}\W_{i_r,i_{r+1}})
\amalg (\coprod_{r=0}^{k-1} \R_{i_r})$.

We also have that
$^{\eperp}{\rigid}_{\T}  \cap \W_{i_r,i_{r+1}} = {^{\eperp}{\rigid}_r}  \cap \W_{i_r,i_{r+1}}$,
which gives us the following description of $\Free$:
$$\Free = (\coprod_{r=0}^{k-1} (^{\eperp}{\rigid}_r)\cap \T)
\amalg (\coprod_{r=0}^{k-1} \R_{i_r}).$$
Since ${\rigid}_r$ is cotilting in $\W_{i_r,i_{r+1}}$, we have that
$^{\eperp}{\rigid}_r  \cap \W_{i_r,i_{r+1}} = \Cogen_{\W_{i_r,i_{r+1}}} {\rigid}_r$.
From this it follows that
$$\Free = (\coprod_{r=0}^{k-1} \Cogen_{\W_{i_r,i_{r+1}}} {\rigid}_r)
\amalg \coprod_{r=0}^{k-1} \R_{i_r}.$$
Noting that $\Cogen M_{i,\infty} \cap \T = \R_i$,
we see that $\Free = \Cogen {\rigid} \cap \T$ as claimed in (a).

Next, we consider slightly larger wings, $\W_{i_r,i_{r+1}+1}$ in order to
obtain the description of $\Free$ in (b).
For each $r$, define $\widetilde{{\rigid}}_r = {\rigid}_r \amalg M_{i_r, i_{r+1}+1}$. Then
$\widetilde{{\rigid}}_r$ is a cotilting module in $\W_{i_r, i_{r+1}+1}$.
Note that if $\rigid$ has exactly one Pr\"{u}fer summand, we need to
regard $\W_{i_r,i_{r+1}+1}=\W_{i_0,i_0+n+1}$ as being a type $A$ category non-exactly embedded in the tube, since the image of the projective-injective object is not rigid.

Note that
$$\Cogen_{\W_{i_r,i_{r+1}+1}}\widetilde{{\rigid}}_r=(\Cogen_{\W_{i_r,i_{r+1}}}{\rigid}_r)\amalg
\add M_{i_r,i_{r+1}}.$$
We can thus rewrite $\Free$ as follows:
$$\Free = (\coprod_{r=0}^{k-1} \Cogen_{\W_{i_r,i_{r+1}+1}} \widetilde{{\rigid}}_r)
\amalg (\coprod_{r=0}^{k-1} \R_{i_r}).$$
So, setting $\Free_r=\Cogen_{\W_{i_r,i_{r+1}+1}} \widetilde{{\rigid}}_r$, we obtain a
description of $\Free$ as claimed in (b).

Next, we compute $\Torsion$.
By definition, we have that $\Torsion= {^{\hperp}\Free} \cap \T$. Furthermore,
$$^{\hperp}( \amalg_{r=0}^{k-1} M_{i_r,i_{r+1}}) \cap \T= (\coprod_{r=0}^{k-1}
\tau^{-1} \W_{i_r,i_{r+1}}) \amalg (\coprod_{r_0}^{k-1} \C_{i_{r+1} +1}).$$

Noting that
$\Hom(M_{l,i_{r+1}+1}, M_{i_r,i_{r+1}+1}) \not= 0$ for $l \leq i_r$, we see
that
$$\Torsion \subset \coprod_{r=0}^{k-1} \tau^{-1} (\W_{i_r,i_{r+1}}) \subset \coprod_{r=0}^{k-1} \W_{i_r,i_{r+1}+1}.$$
From this, it follows that
$\Torsion = \coprod_{r=0}^{k-1}(^{\hperp}\Free_r \cap \W_{i_r,i_{r+1}+1})$, noting
that $$\Free_r= \Cogen_{\W_{i_r,i_{r+1}+1}} \widetilde{{\rigid}}_r.$$
Using the fact that $\widetilde{{\rigid}}_r$ is cotilting in
$\W_{i_r,i_{r+1}+1}$, combined with Corollary~\ref{cor:Gen-Cogen},
we see that
$\Torsion = \coprod_{r=0}^{k-1}\Torsion_r$, where
$$\Torsion_r = {^{\hperp}\widetilde{{\rigid}}_r} \cap
\W_{i_r,i_{r+1}+1} = \Gen_{\W_{i_r, i_{r+1}+1}} \tau_{\W_{i_r,i_{r+1}+1}}^{-1} \widetilde{{\rigid}}_r$$
is the torsion part of the torsion pair in $\W_{i_r,i_{r+1}+1}$ with torsion-free
part $\Free_r$.
We have that $^{\hperp}(\amalg_{r=0}^{k-1} M_{i_r, \infty}) =
\coprod_{r=0}^{k-1} \W_{i_r, i_{r+1}}$, and therefore $\Torsion= {^{\hperp}{\rigid}} \cap
\T$.

Since $\Gen M_{i, \infty} \cap \T$ is zero for all $i$ and
since we have
$\tau_{\W_{i_r,i_{r+1}+1}}^{-1} \widetilde{{\rigid}}_r  = \tau_{\W_{i_r,i_{r+1}+1}}^{-1} {\rigid}_r=
\tau^{-1} {\rigid}_r$, we see that
$\Torsion = \tau^{-1} {\Gen {\rigid} \cap \T}$.

See Figure~\ref{fig:wing} for an illustration of one of the wings
$\W_{i_r,i_{r+1}+1}$. We have that $\Free_r=\Cogen_{\W_{i_r,i_{r+1}+1}}\widetilde{{\rigid}}_r$.
The indecomposable objects in $\Free_r$ apart from $M_{i_r,i_{r+1}+1}$ lie
in the left hand shaded triangle. Also, $\Torsion_r=\Gen_{\W_{i_r,i_{r+1}+1}}\tau_{\W_{i_r,i_{r+1}+1}}^{-1}\widetilde{{\rigid}}_r$,
and the indecomposable objects in $\Torsion_r$ lie in the right hand shaded triangle.
\begin{figure}
\psfragscanon
\psfrag{M1}{$\scriptstyle M_{i_r,i_r+2}$}
\psfrag{M2}{$\scriptstyle M_{i_r+1,i_r+3}$}
\psfrag{M3}{$\scriptstyle M_{i_{r+1}-2,i_{r+1}}$}
\psfrag{M4}{$\scriptstyle M_{i_{r+1}-1,i_{r+1}+1}$}
\psfrag{M5}{$\scriptstyle M_{i_r,i_{r+1}}$}
\psfrag{M6}{$\scriptstyle M_{i_r,i_{r+1}+1}$}
\psfrag{M7}{$\scriptstyle M_{i_r,i_{r+1}+1}$}
\includegraphics[width=7cm]{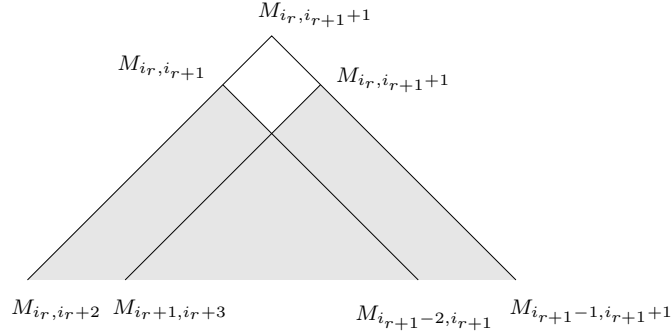}
\caption{A wing $\W_{i_r,i_{r+1}+1}$.}
\label{fig:wing}
\end{figure}
\end{proof}

\subsection{The maximal rigid module corresponding to a torsion pair}
In this section we give an explicit description of the inverse of the bijection
${\rigid}\mapsto (\Torsion_{\rigid},\Free_{\rigid})$ in Theorem~\ref{t:mainbijection} between maximal rigid
objects in $\overline{\T}$ and torsion pairs in $\T$.

\begin{lemma} \label{l:limG}
Let $(\Torsion,\Free)$ be a torsion pair in $\T$. Then an indecomposable object
in $\varinjlim(\ind \Free)$ either lies in $\ind\Free$
or it is a Pr\"{u}fer module which is the direct limit along a ray in
$\ind \Free$.
\end{lemma}
\begin{proof}
It is clear that any indecomposable in $\Free$ or Pr\"{u}fer module which
is the direct limit along a ray in $\ind \Free$ lies in
$\ind(\varinjlim \Free)$.
So suppose that $X=\varinjlim X_j$ is an indecomposable
in $\overline{\T}$, where each $X_j$ is in $\ind \Free$.
Firstly note that $X$ cannot be an adic module, since the adic modules
do not lie in $\varinjlim \T$.
Next, for any object ${\torsion}$ in $\Torsion$, we have
$$\Hom({\torsion},X)=\Hom({\torsion},\varinjlim X_j)\simeq \varinjlim \Hom({\torsion},X_j)=0,$$
using Lemma~\ref{limits}(b), since each $X_j$ lies in $\Free$.
It follows that if $X$ is in $\ind \T$, we have that $X$ lies in $\ind\Free$.
The only other possibility is if $X$ is a Pr\"{u}fer module. If $Y$ is
any indecomposable object in the corresponding ray in $\T$,
we have an embedding $Y\rightarrow X$. Hence $\Hom({\torsion},Y)=0$ for any object
${\torsion}$ in $\Torsion$, so $Y$ lies in $\Free$. The result follows.
\end{proof}

\begin{proposition} \label{p:TfromFG}
\begin{enumerate}
\item[(a)] Let $(\Torsion,\Free)$ be a torsion pair of ray type. Then $(\Torsion,\Free)$ can
be written in the form $(\Torsion_{\rigid},\Free_{\rigid})$, where ${\rigid}$
is the direct sum of the indecomposable objects in
\begin{equation}
\varinjlim (\ind \Free) \cap (\varinjlim (\ind \Free))^{\eperp}.
\label{e:rayinverse}
\end{equation}
Compare~\cite[Thm.\ 1.5]{bk1}.
\item[(b)] Let $(\Torsion,\Free)$ be a torsion pair of coray type. Then $(\Torsion,\Free)$ can
be written in the form $(\Torsion_{\rigid},\Free_{\rigid})$, where ${\rigid}$
is the direct sum of the indecomposable objects in
\begin{equation}
\varprojlim (\ind \Torsion) \cap {}^{\eperp}(\varprojlim (\ind \Torsion)).
\label{corayinverse}
\end{equation}
\end{enumerate}
\end{proposition}

\begin{proof}
We only consider (a); the proof of (b) is similar.
By Theorem~\ref{t:mainbijection} and Proposition~\ref{p:ftgt}(b) and its proof,
there is a maximal rigid object ${\rigid}$ in $\overline{\T}$ of Pr\"{u}fer type such
that $(\Torsion,\Free)=(\Torsion_{\rigid},\Free_{\rigid})$ and $\Torsion_{\rigid}$, $\Free_{\rigid}$ have the following description.
Let the Pr{\"u}fer summands of ${\rigid}$ be $M_{i_r,\infty}$
for $r=0, \dots, k-1$ where $0 \leq i_0 < i_1 < \cdots < i_{k-1}
\leq n-1$. Then $$(\Torsion_{\rigid}, \Free_{\rigid}) = (\coprod_{r=0}^{k-1}
  \Torsion_r, (\coprod_{r=0}^{k-1}\Free'_r) \amalg \Free_{\infty}),$$ where
$\Torsion_r,\Free_{\infty}$ are as in Proposition~\ref{p:ftgt}(b) and
$\Free'_r=\Cogen_{\W_{i_r,i_{r+1}}} {\rigid}_r$, where ${\rigid}|_{\T}=\amalg_{r=0}^{k-1} {\rigid}_r$
with ${\rigid}_r$ in $\W_{i_r,i_{r+1}}$.

By~\eqref{e:extprojectives} in Section~\ref{s:Dynkin},
\begin{align*}
\ind {\rigid}_r &=  \{X\in \ind\Free'_r\,:\, \Ext^1(Y,X)=0\text{\ for all\ } Y\in \Free'_r\} \\
&= \{X\in \ind\Free'_r \,:\, \Ext^1(Y,X)=0\text{\ for all\ } Y\in \ind\Free\}.
\end{align*}
Note that $\Free_{\infty}=\coprod_{r=0}^{k-1} \R_{i_r} \subset \ind\Free$.
It follows that any indecomposable object $X$ in $\T$ which satisfies $\Ext^1(Y,X)=0$
for all $Y$ in $\ind \Free$ must lie inside some wing $\W_{i_r,i_{r+1}}$,
and hence in some $\Free'_r$. Therefore
\begin{align*}
\ind {\rigid}|_{\T} &= \ind(\amalg_{r=0}^{k-1} {\rigid}_r) \\
&= \{X\in \ind \Free\,:\,\Ext^1(Y,X)=0\text{\ for all\ } Y\in \ind\Free\}.
\end{align*}

Next, by Lemma~\ref{l:limG}, an indecomposable object in
$\varinjlim(\ind \Free)$ either lies in $\ind\Free$ or is the direct limit
along a ray in $\Free$, i.e.\ it is a Pr\"{u}fer summand $M_{i_r,\infty}$ of ${\rigid}$.
If $Y$ is one of the latter summands and $X$ lies in $\ind\Free$, we have
$D\Ext^1(Y,X)\simeq \Hom(X,\tau Y)=0$,
by Lemmas~\ref{ARformula} and~\ref{homvanishing}. Since the Pr\"{u}fer summands
themselves satisfy
$\Ext^1(M_{i_r,\infty},M_{i_{r'},\infty})=0$
for all $r,r'$, the result follows.
\end{proof}

\section{Geometric interpretation in the tube case}
\label{s:tubegeometric}
\subsection{Short exact sequences}

Let $[i,j],[i',j']$ be arcs in $\A(\UUU(n))$, giving
rise to arcs $\pi_n([i,j]),\pi_n([i',j'])$ in $\A(\AAA(n))$ with corresponding
indecomposable objects $M_{ij},M_{i'j'}$ in ${\rigid}$.
Suppose that $I^-_{\UUU(n)}([i,j],[i'+kn,j'+kn])$ (i.e.\ the intersection
number between arcs $[i,j]$ and $[i'+kn,j'+kn]$ in $\UUU(n)$) is equal to $1$ for some
$k$ in $\mathbb{Z}$. Then, if $j'+kn>i+1$, we have four objects in the Auslander-Reiten quiver of $\T$ as shown in Figure~\ref{fig:diamond}(a); if $j'+kn=i+1$, we have three
objects as shown in Figure~\ref{fig:diamond}(b).

\begin{figure}
\centering
\subfigure[Case $j'+kn>i+1$.]{
$$\begin{xy} *!D\xybox{
\xymatrix@!0@=1.2cm{
& M_{i'+kn,j} \ar[dr] & \\
M_{i'j'}=M_{i'+kn,j'+kn} \ar[dr] \ar[ur] && M_{ij} \\
& M_{i,j'+kn} \ar[ur]
}}\end{xy}$$}
\subfigure[Case $j'+kn=i+1$.]{
$$$$\begin{xy} *!D\xybox{
\xymatrix@!0@=1.2cm{
& M_{i'+kn,j} \ar[dr] & \\
M_{i'j'}=M_{i'+kn,j'+kn} \ar[ur] && M_{ij} \\
&
}}\end{xy}$$}
\caption{Objects and paths in the Auslander-Reiten quiver of $\T$ corresponding
 to an intersection in $\UUU(n)$ between arcs $[i,j]$ and $[i'+kn,j'+kn]$.}
\label{fig:diamond}
\end{figure}
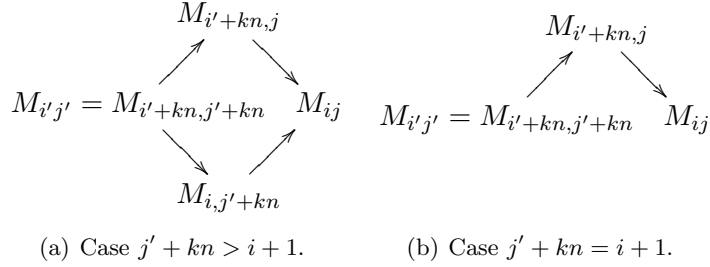

In the case $j'+kn>i+1$, this corresponds to a non-split short exact sequence:
\begin{equation}
0 \rightarrow M_{i',j'} \overset{f}{\longrightarrow} M_{i'+kn,j}\amalg M_{i,j'+kn}
\overset{g}{\rightarrow} M_{ij}\rightarrow 0,
\label{e:diamondses}
\end{equation}
and in the case $j'+kn=i+1$, it corresponds to a short exact sequence:
\begin{equation}
0 \rightarrow M_{i',j'} \overset{f}{\longrightarrow} M_{i'+kn,j}\overset{g}{\rightarrow} M_{ij}\rightarrow 0
\label{e:diamondsesb}
\end{equation}
in $\T$. As in~\cite[Remark 4.25]{warkentin},
these can be interpreted geometrically in $\UUU(n)$: see Figure~\ref{fig:tubesplit-sum}.

\begin{figure}
\psfragscanon
\psfrag{i'+kn}{$\scriptstyle i'+kn$}
\psfrag{i}{$\scriptstyle i$}
\psfrag{j'+kn}{$\scriptstyle j'+kn$}
\psfrag{j}{$\scriptstyle j$}
\subfigure[Case $j'+kn>i+1$.]{\includegraphics[width=3.2cm]{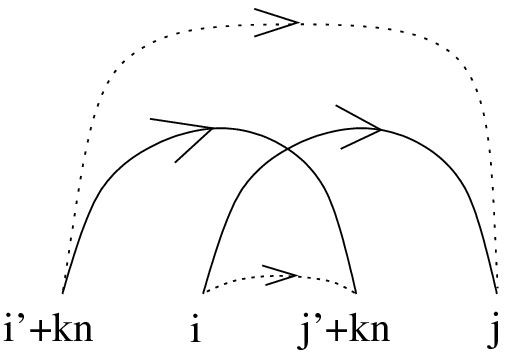}}
\quad \quad \quad
\subfigure[Case $j'+kn=i+1$.]{\includegraphics[width=3.25cm]{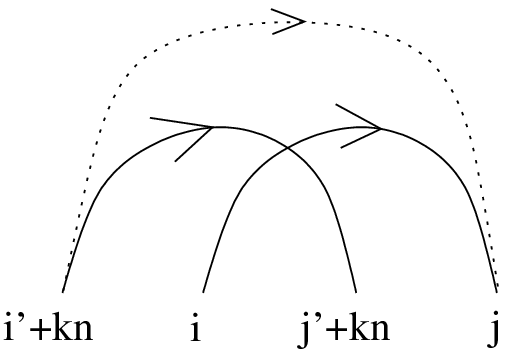}}
\caption{Non-split extensions in $\T$ represented geometrically.}
\label{fig:tubesplit-sum}
\end{figure}

If $j'+kn>i+1$, write $f=\begin{pmatrix} f_1 \\ f_2 \end{pmatrix}$ and
$g=\begin{pmatrix} g_1 & g_2 \end{pmatrix}$.
Then $f_1$ and $g_2$ are monomorphisms and $f_2$ and $g_1$ are epimorphisms.
The $f_i$ and $g_i$ are uniquely determined up to a choice
of scalars. If $j'+kn=i+1$, $f$ is a monomorphism and $g$ is an epimorphism,
again uniquely determined up to a choice of scalars.

\begin{lemma}
Any non-split short exact sequence with first term $M_{i'j'}$ and last
term $M_{ij}$ has the same form as~\eqref{e:diamondses} or~\eqref{e:diamondsesb}.
\end{lemma}

\begin{proof}
Let
\begin{equation}
0\rightarrow M_{i'j'} \overset{u}{\rightarrow} E \overset{v}{\rightarrow} M_{ij} \rightarrow 0
\label{e:typicalses}
\end{equation}
be an arbitrary non-split short exact sequence with first term $M_{i'j'}$ and
last term $M_{ij}$.
By Lemma~\ref{l:uniserialclosure} we may write $E=E_1\amalg E_2$ where $E_1$ is
indecomposable and such that we have that, decomposing $u=\begin{pmatrix} u_1 \\ u_2 \end{pmatrix}$ and $v=(v_1,v_2)$, we have that $u_1$ a monomorphism.
Since the sequence is not split, $u_1$ is not an isomorphism, so, denoting the length
of an object $M$ in $\T$ by $\ell(M)$, we have
$$\ell(E_2)=\ell(M_{ij})+\ell(M_{i'j'})-\ell(E_1)<\ell(M_{ij}).$$
Since $v_2$ is not an epimorphism, $v_1$ must be an epimorphism, again using
Lemma~\ref{l:uniserialclosure}.

\emph{Case (i)}: Suppose first that $v_1u_1\not=0$.
If an integer $k$ is such that $i'+kn\leq i$ and $i+2\leq j'+kn\leq j$,
there is a homomorphism in $\T$ from $M_{i'j'}$ to $M_{ij}$ obtained by
composing an epimorphism from $M_{i'j'}$ to $M_{i,j'+kn}$ with a monomorphism
from $M_{i,j'+kn}$ to $M_{ij}$ (i.e.\ the two maps in the lower edges of the
diamond in Figure~\ref{fig:diamond}).

It is easy to check that the homomorphisms of this kind (allowing $k$ to vary)
form a basis of $\Hom(M_{i'j'},M_{ij})$. Since $vu=0$ and $v_1u_1\not=0$ and
$v_1u_1$ is a scalar multiple of such a basis element, there must be an indecomposable
summand $X$ of $E_2$ such that $v_Xu_X$ is a scalar multiple of $v_1u_1$
(where $u_X$, $v_X$ are the corresponding components of $u_2$, $v_2$).
But
\begin{equation}
\ell(X)\leq \ell(M_{ij})+\ell(M_{i'j'})-\ell(E_1),
\label{e:lowpoint}
\end{equation}
and there is a unique path from $M_{i'j'}$ to $M_{ij}$ through
such an $X$ giving rise to $g_1f_1$
(i.e.\ with $X=M_{i,j'+kn}$) from which it follows that
we have equality in~\eqref{e:lowpoint} and thus that $E_2=X$ is
indecomposable and $u_2$ is an epimorphism and $v_2$ is a monomorphism.
It follows that the short exact sequence~\eqref{e:typicalses} is of
the form~\eqref{e:diamondses} up to a choice of scalars.

\emph{Case (ii)}: Now assume that $v_1u_1=0$. This implies that
$\ell(E_1)\geq \ell(M_{ij})+\ell(M_{i'j'})$,
but we also have $\ell(E_1)\leq \ell(E_1\amalg E_2)=\ell(M_{ij})+\ell(M_{i'j'})$,
so we must have equality and $E=E_1$ is indecomposable.
It follows that~\eqref{e:typicalses} is of the form~\eqref{e:diamondsesb} up
to a choice of scalars. The proof is complete.
\end{proof}

\begin{definition}
We call a collection $\SS$ of arcs in $\A(\AAA(n))$ an \emph{oriented Ptolemy
diagram} (in $\AAA(n)$) if, whenever $\pi_n([i,j])$ and $\pi_n([i',j'])$ lie in
$\SS$ with $i'<i<j'<j$ then $\pi_n([i,j'])$ (when $j'>i+1$) and $\pi_n([i',j])$
also lie in $\SS$ (see related definitions in Section~\ref{s:geometricmodel} and~\cite{ng,hjr}).
\end{definition}

We note that the additive closure of a collection of indecomposable objects in $\T$ is closed
under extensions if and only if the corresponding collection of arcs is an oriented Ptolemy diagram in $\AAA(n)$. It is also easy to check that if $\pi_n([i,j])$
is an arc in $\A(\AAA(n))$, then the indecomposable quotients of $M_{ij}$ are the $M_{i'j}$ where
$i\leq i'\leq j-2$, i.e.\ arcs in $\AAA(n)$ corresponding to arcs in $\UUU(n)$ with the same ending point and with starting point weakly to the right of $i$.
Call these the \emph{left-shortenings} of $\pi_n([i,j])$. Similarly, submodules
are given by right-shortenings.

We make the following remark, which follows from Proposition~\ref{crossingnumbers}.
Recall that $\widetilde{\A}(\AAA(n))$ denotes $\A = \A(\AAA(n))$ extended to include the homotopy classes of the arcs $\pi_n([i,\infty])$ and $\pi_n([-\infty,i])$.

\begin{remark}
The bijection $\psi:\pi_n([i,j])\mapsto M_{ij}$ between $\widetilde{\A}(\AAA(n))$ and $\ind(\overline{\T})$
induces a bijection between maximal noncrossing collections of arcs in $\AAA(n)$ (including
the infinite arcs) and maximal rigid objects in $\overline{\T}$.
\end{remark}

We can now describe the conditions on collections of arcs appearing in torsion
pairs in $\T$, using the above and Lemmas~\ref{l:tpclosure} and~\ref{l:uniserialclosure}.

\begin{proposition}
\begin{enumerate}
\item[(a)]
A collection $\SS$ of arcs in $\A(\AAA(n))$ corresponds to the torsion part
of a torsion pair in $\T$ if and only if
$\SS$ is an oriented Ptolemy diagram in $\AAA(n)$ and $\SS$ is closed under left-shortening.
\item[(b)]
A collection $\SS$ of arcs in $\A(\AAA(n))$ corresponds to the torsion-free part
of a torsion pair in $\T$ if and only if $\SS$ is an oriented Ptolemy diagram in $\AAA(n)$
and $\SS$ is closed under right-shortening.
\end{enumerate}
\end{proposition}

We remark that, if ${\rigid}$ is of Pr\"{u}fer type, by Proposition~\ref{p:ftgt},
$\psi(\ind\Torsion_{\rigid})$ can be obtained by taking the closure of the set of arcs
corresponding to finitely generated indecomposable summands of ${\rigid}$ under left
shortening and rotating all resulting arcs one step to the right.
We obtain $\psi(\ind\Free_{\rigid})$ by taking the closure of the set of arcs
corresponding to \emph{all} of the summands of ${\rigid}$ under right shortening.

Similarly, by Proposition~\ref{p:TfromFG}, if $(\Torsion,\Free)$ is a torsion pair where
$\Free$ generates $\T$ (i.e.\ of ray type), then $\psi(\varinjlim(\ind\Free))$ can be obtained from
$\psi(\ind\Free)$ by first adding any infinite arc $\pi_n([i,\infty])$ for which all
arcs $\pi_n([i,j])$ for $j\geq a$ for some $a$ lie in $\psi(\ind\Free)$.
Then ${\rigid}$ is the direct sum of the indecomposable objects corresponding to the
arcs $\alpha$ in $\psi(\varinjlim(\ind\Free))$ such that the pair
$(\beta,\alpha)$ of arcs has no negative intersections for all
$\beta$ in $\psi(\varinjlim(\ind\Free))$.

Similar descriptions can be given in the adic/coray type case.

Finally, we give an example in a tube of rank $n=14$ to illustrate Proposition~\ref{p:ftgt} and the results in this section.
The arcs corresponding to the indecomposable direct summands of ${\rigid}$ are
displayed in Figure~\ref{fig:tubeex}
(only the beginnings of the infinite arcs are shown). Note that the Pr\"{u}fer modules
which are indecomposable summands of ${\rigid}$ are $M_{0,\infty},M_{6,\infty},M_{10,\infty}$ and
$M_{13,\infty}$, so $i_0=0$, $i_1=6$, $i_2=10$ and $i_3=13$.
\begin{figure}[ht]
\psfragscanon
\psfrag{0}{$0$}
\psfrag{1}{$1$}
\psfrag{2}{$2$}
\psfrag{3}{$3$}
\psfrag{4}{$4$}
\psfrag{5}{$5$}
\psfrag{6}{$6$}
\psfrag{7}{$7$}
\psfrag{8}{$8$}
\psfrag{9}{$9$}
\psfrag{10}{$10$}
\psfrag{11}{$11$}
\psfrag{12}{$12$}
\psfrag{13}{$13$}
\includegraphics[width=10cm]{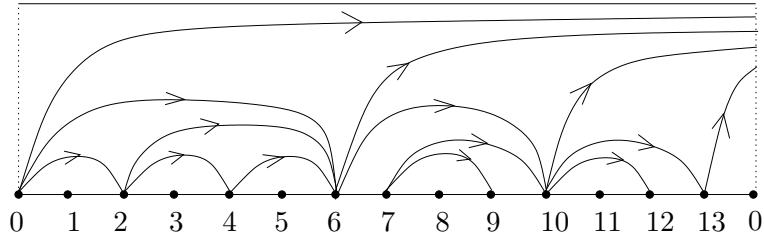}
\caption{The maximal rigid object ${\rigid}$}
\label{fig:tubeex}
\end{figure}
The arcs corresponding to the indecomposable objects in $\Torsion_{\rigid}$ are displayed in
Figure~\ref{fig:tubeF}.
\begin{figure}[ht]
\psfragscanon
\psfrag{0}{$0$}
\psfrag{1}{$1$}
\psfrag{2}{$2$}
\psfrag{3}{$3$}
\psfrag{4}{$4$}
\psfrag{5}{$5$}
\psfrag{6}{$6$}
\psfrag{7}{$7$}
\psfrag{8}{$8$}
\psfrag{9}{$9$}
\psfrag{10}{$10$}
\psfrag{11}{$11$}
\psfrag{12}{$12$}
\psfrag{13}{$13$}
\includegraphics[width=10cm]{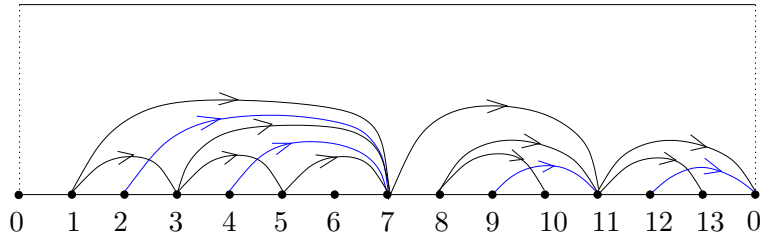}
\caption{The torsion part, $\Torsion_{\rigid}=\tau^{-1}(\Gen {\rigid}\cap \T)$, of the torsion pair corresponding to ${\rigid}$.
The arcs corresponding to indecomposable objects not in $\tau^{-1}(\add {\rigid}\cap \T)$ are drawn in blue.}
\label{fig:tubeF}
\end{figure}

The arcs corresponding to the indecomposable objects in the $\Free_{\rigid}\cap \W_{i_r,i_{r+1}+1}$
are displayed in Figure~\ref{fig:tubeG}
(with dotted arcs indicating the indecomposable
summands of ${\rigid}$ which are not in $\T$ (or $\ind\Free$)).
Note that there are infinitely many additional arcs not displayed, corresponding to indecomposables in $\Free_{\infty}$ but not in any of the $\Free_r$.
The missing arcs are $\pi_{14}([0,j])$ for $j\geq 8$, $\pi_{14}([6,j])$ for $j\geq 12$,
$\pi_{14}([10,j])$ for $j\geq 15$ and $\pi_{14}([13,j])$ for $j\geq 16$.
Note that, as indicated by Proposition~\ref{p:ftgt}, the intersections of
$\Torsion_{\rigid}$ and $\Free_{\rigid}$ with the wings $\W_{i_r,i_{r+1}+1}$ (which are
$\W_{0,7},\W_{6,11},\W_{10,14}$ and $\W_{13,15}$) are torsion pairs.

\begin{figure}[ht]
\psfragscanon
\psfrag{0}{$0$}
\psfrag{1}{$1$}
\psfrag{2}{$2$}
\psfrag{3}{$3$}
\psfrag{4}{$4$}
\psfrag{5}{$5$}
\psfrag{6}{$6$}
\psfrag{7}{$7$}
\psfrag{8}{$8$}
\psfrag{9}{$9$}
\psfrag{10}{$10$}
\psfrag{11}{$11$}
\psfrag{12}{$12$}
\psfrag{13}{$13$}
\includegraphics[width=10cm]{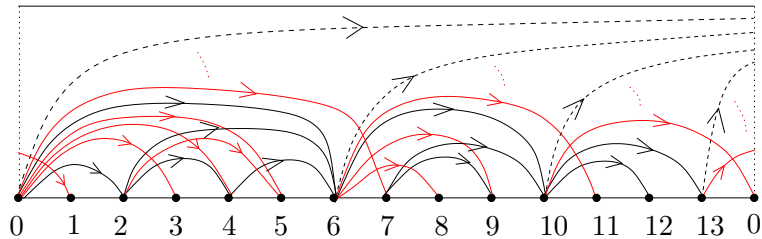}
\caption{The torsion-free part, $\Free_{\rigid}=\Cogen {\rigid}\cap \T$, of the torsion pair corresponding to ${\rigid}$.
The arcs not in $\add {\rigid}$ are drawn in red. The dashed arcs are the indecomposable summands of ${\rigid}$
which are not of finite length (and thus not in $\Free_{\rigid}$).
The arcs $\pi_{14}([0,j])$ for $j\geq 8$, $\pi_{14}([6,j])$ for $j\geq 12$,
$\pi_{14}([10,j])$ for $j\geq 15$ and $\pi_{14}([13,j])$ for $j\geq 16$ have been omitted
for clarity.}
\label{fig:tubeG}
\end{figure}
In Figure~\ref{fig:tubeAR}, we show the indecomposable summands of $\T$ and the indecomposable objects in $\Torsion_{\rigid}$ and $\Free_{\rigid}$ in the AR-quiver of the tube.
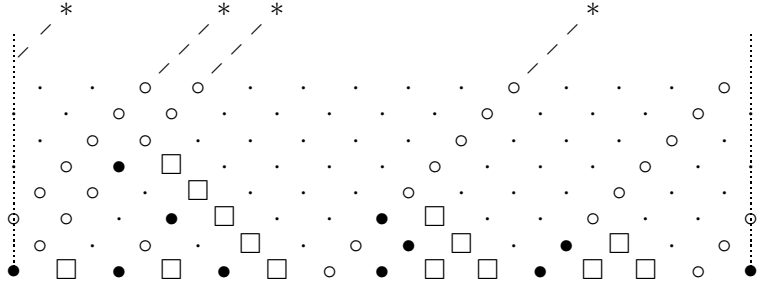
\begin{figure}[ht]
$$\xymatrix@!0@=0.35cm{
&& \ast &&&&&& \ast && \ast &&&&&&&&&&&& \ast &&&&&& \\
&&&&&&&&&&&&&&&&&&&&&&&&&&&& \\
\ar@{--}[uurr] &&&&&&&&&&&&&&&&&&&&&&&&&&&& \\
&\cdot && \cdot && \circ \ar@{--}[uuurrr] && \circ \ar@{--}[uuurrr] && \cdot && \cdot && \cdot && \cdot && \cdot && \circ \ar@{--}[uuurrr] && \cdot && \cdot && \cdot && \circ \ar@{-}[ur] & \\
\cdot && \cdot && \circ && \circ && \cdot && \cdot && \cdot && \cdot && \cdot && \circ && \cdot && \cdot && \cdot && \circ && \cdot && \\
&\cdot && \circ && \circ && \cdot && \cdot && \cdot && \cdot && \cdot && \circ && \cdot && \cdot && \cdot && \circ && \cdot & \\
\cdot && \circ && \bullet && \Box && \cdot && \cdot && \cdot && \cdot && \circ && \cdot && \cdot && \cdot && \circ && \cdot && \cdot && \\
& \circ && \circ && \cdot && \Box && \cdot && \cdot && \cdot && \circ && \cdot && \cdot && \cdot && \circ && \cdot && \cdot & \\
\circ && \circ && \cdot && \bullet && \Box && \cdot && \cdot && \bullet && \Box && \cdot && \cdot && \circ && \cdot && \cdot && \circ && \\
&\circ && \cdot && \circ && \cdot && \Box && \cdot && \circ && \bullet && \Box && \cdot && \bullet && \Box && \cdot && \circ & \\
*=0{\bullet}  \ar@{.}[uuuuuuuuu]+0 && \Box && \bullet && \Box && \bullet && \Box && \circ && \bullet && \Box && \Box && \bullet && \Box && \Box && \circ && *=0{\bullet} \ar@{.}[uuuuuuuuu]+0 &&
}$$
\caption{The AR-quiver of the tube, showing the indecomposable summands of ${\rigid}$ ($\bullet$ or $\ast$), $\ind(\Torsion_{\rigid})$ ($\Box$) and $\ind(\Free_{\rigid})$ ($\circ$ or $\bullet$). The Pr\"{u}fer direct summands of ${\rigid}$ are shown (symbolically) at the top of the diagram as asterisks.}
\label{fig:tubeAR}
\end{figure}

\noindent \textbf{Acknowledgements}
All three authors would like to thank the referee for helpful comments on an
earlier version of this manuscript, and the Mathematics Research Institute in
Oberwolfach for its support during a conference in February 2011.
ABB would also like to thank Karin Baur and the FIM at ETH for their support
and kind hospitality during a visit in May 2011.
KB would like to thank Aslak Buan and the NTNU for their kind hospitality
during a visit in December 2010.
RJM would like to thank Karin Baur and the FIM at the ETH, Zurich, for their
support and kind hospitality during a visit in Spring 2011,
and would like to thank Andrew Hubery for a helpful conversation.

\end{document}